\documentclass[reqno,11pt]{amsart}

\pdfoutput=1

\usepackage{pinlabel}
\input{epsf.tex}
\usepackage{latexsym,amsfonts,amssymb,verbatim,mathrsfs,amsthm}
\usepackage{amsmath,amsthm,amssymb,latexsym,graphics,textcomp}
\usepackage{eucal,eufrak}
\usepackage{colonequals}
\usepackage{graphicx}
\usepackage{url}
\input{xy}
\xyoption{all}

\theoremstyle{plain}
\newtheorem{theorem}{Theorem}[section]
\newtheorem{lemma}[theorem]{Lemma}
\newtheorem{corollary}[theorem]{Corollary}

\newtheorem{definition}[theorem]{Definition}

\newtheorem{proposition}[theorem]{Proposition}

\theoremstyle{definition}
\newtheorem{remark}{Remark}

\def\1{\mathbf 1}

\def\Mod{\mathrm{Mod}}

\def\Aut{\mathrm{Aut}}

\def\C{\mathcal C}

\def\XX{\mathfrak X}

\def\CC{\mathfrak C}
\def\OO{\mathfrak O}

\def\BB{\mathfrak B}

\title{Exhausting curve complexes by finite rigid sets}
\author{Javier Aramayona and Christopher J. Leininger}
\thanks{The first author was supported by BQR (Toulouse) and Campus Iberus grants. The second author was supported by NSF grants DMS 0905748 and 1207183.}
\begin{document}

\begin{abstract}
Let $S$ be a connected orientable surface of finite topological type. We prove that  there is an exhaustion of the curve complex $\C(S)$ by a sequence of finite rigid sets.
\end{abstract}

\maketitle

\section{Introduction}
The curve complex $\C(S)$ of a surface $S$ is a simplicial complex whose $k$-simplices correspond to sets of $k+1$ distinct isotopy classes of essential simple closed curves  on $S$ with pairwise disjoint representatives. The extended mapping class group $\Mod^{\pm}(S)$ of $S$ acts on $\C(S)$ by simplicial automorphisms, and a  well-known theorem due to Ivanov \cite{Ivanov}, Korkmaz \cite{Korkmaz} and Luo \cite{Luo}, asserts that $\C(S)$ is {\em simplicially rigid} for $S\ne S_{1,2}$. More concretely, the natural homomorphism \[\Mod^{\pm}(S) \to \Aut(\C(S))\] is surjective unless $S=S_{1,2}$; in the case $S=S_{1,2}$ there is an automorphism of $\C(S)$ that sends a separating curve on $S$ to a non-separating one and thus cannot be induced by an element of $\Mod^{\pm}(S)$, see \cite{Luo}. 

In \cite{AL} we extended this picture and showed that curve complexes are {\em finitely rigid}. Specifically, for $S\ne S_{1,2}$ we identified a finite subcomplex $\XX(S) \subset \C(S)$ with the property that very locally injective map $\XX(S) \to \C(S)$ is the restriction of an element of $\Mod^{\pm}(S)$; in the case of $S_{1,2}$ a similar statement can be made, this time using the group $\Aut(\C(S))$ instead of $\Mod^{\pm}(S)$. We refer to such a subset $\XX(S)$  as a {\em rigid} set.

The rigid sets  constructed in \cite{AL} enjoy some curious properties. For instance, if $S=S_{0,n}$ is a sphere with $n$ punctures then $\XX(S)$ is a homeomorphic to an $(n-4)$-dimensional sphere. Since $\C(S)$ has dimension $n-4$, it follows that $\XX(S)$ represents a non-trivial element of $H_{n-4}(\C(S), \mathbb Z)$ which, by a result of Harer \cite{Harer}, is the only non-trivial homology group of $\C(S)$. In fact,  Broaddus \cite{Broaddus} and Birman-Broaddus-Menasco \cite{BBM} have recently proved $\XX(S)$ is a generator of  $H_{n-4}(\C(S), \mathbb Z)$, when viewed as a $\Mod^{\pm}(S)$-module; in the case when $S$ has  genus $\ge 2$ and at least one puncture, they prove that $\XX(S)$ {\em contains} a generator for the homology of $\C(S)$.

The rigid sets identified in \cite{AL} all have diameter 2 in $\C(S)$, and a natural question is whether there exist finite rigid sets in $\C(S)$ of arbitrarily large diameter;  see Question 1 of \cite{AL}.  In this paper we  prove that, in fact,  there exists an exhaustion of $\C(S)$ by finite rigid sets: 

\begin{theorem}
Let $S\ne S_{1,2}$ be a connected orientable surface of finite topological type. There exists a sequence  $\XX_1 \subset \XX_2 \subset \ldots \subset \C(S)$ 
such that: 
\begin{enumerate}
\item  $\XX_i$ is a finite rigid set for all $i\ge 1$, 
\item $\XX_i$ has trivial pointwise stabilizer in $\Mod^{\pm}(S)$, for all $i\ge 1$, and
\item  $\bigcup_{i\ge 1} \XX_i = \C(S)$.
\end{enumerate}
\label{main}
\end{theorem}

\begin{remark}
A similar statement can be made for $S= S_{1,2}$, by replacing $\Mod^{\pm}(S)$ by $\Aut(\C(S))$ in the definition of rigid set above. 
\end{remark}

 \begin{remark}
We stress that Theorem \ref{main} above does not follow from the main result in \cite{AL}. Indeed, a subset of $\C(S)$ containing a rigid set need not be itself rigid; compare with Proposition \ref{onecurve} below. 
\end{remark}

As a consequence of Theorem  \ref{main} we will obtain a ``finitistic" proof of the aforementioned result of Ivanov-Korkmaz-Luo \cite{Ivanov,Korkmaz,Luo} on the simplicial rigidity of the curve complex. In fact, we will deduce the following stronger form due to Shackleton \cite{Shackleton}:

\begin{corollary}\label{C:Ivanov}
Let $S\ne S_{1,2}$ be a connected orientable surface of finite topological type. If $\phi:\C(S) \to \C(S)$ is a locally injective simplicial map, then there exists $h\in \Mod^{\pm}(S)$ such that $h = \phi$. 
\end{corollary}

The first author and Souto \cite{AS} proved that if $\XX\subset \C(S)$ is a rigid set satisfying some extra conditions, then every (weakly) injective homomorphism from the right-angled Artin group $\mathbb A(\XX)$ into $\Mod^{\pm}(S)$ is obtained, up to conjugation, by taking powers of roots of Dehn twists in the vertices of $\XX$. Since the finite rigid sets $\XX_i$ of Theorem \ref{main} all satisfy the conditions of \cite{AS}, we obtain the following result; here, $T_\gamma$ denotes the Dehn twist about $\gamma$:

\begin{corollary}
Let $S\ne S_{1,2}$ be a connected orientable surface of finite topological type, and consider the sequence  $\XX_1 \subset \XX_2 \subset \ldots \subset \C(S)$ of finite rigid sets given by Theorem \ref{main}. If  $\rho_i:\mathbb A(\XX_i) \to \Mod^{\pm}(S)$ is an injective homomorphism, then there exist functions $a,b:\C^{(0)}(S)\to\ \mathbb Z \setminus\{0\}$ and $f_i \in \Mod^{\pm} (S)$ such that  \[\rho_i(\gamma^{a(\gamma)}) =  f_iT_\gamma^{b(\gamma)}f_i^{-1},\] for every vertex $\gamma$ of $\XX_i$. 
\end{corollary}

We remark that Kim-Koberda \cite{KK1} had previously shown the existence of injective homomorphisms $\mathbb A(Y_i) \to \Mod^{\pm}(S)$ for  sequences $Y_1 \subset Y_2 \ldots $ of subsets of $\C(S)$. Such homomorphisms may in fact be obtained by sending a generator of $\mathbb A(Y_i)$ to a sufficiently high power of a Dehn {\em multi-twist}, see \cite{KK2}. 

\medskip

\noindent{\bf Plan of the paper.} In Section 2 we recall some necessary definitions and basic results from our previous paper \cite{AL}. Section 3 deals with the problem of enlarging a rigid set in a way that remains rigid. As was the case in \cite{AL}, the techniques used in the proof of our main result differ depending on the genus of $S$. As a result, we will prove Theorem \ref{main} for surfaces of genus $0$, $g\ge 2$, and 1, in Sections  4, 5 and 6, respectively.

\medskip

\noindent{\bf Acknowledgements} We thank Brian Bowditch for conversations and for his continued interest in this work. We also thank Juan Souto for conversations. Parts of this paper were completed during the conference ``Mapping class groups and Teichm\"uller Theory"; we would like to  express our gratitude to the Michah Sageev and the Technion for their hospitality and financial support.

\section{Definitions}

Let $S=S_{g,n}$ be an orientable surface of genus $g$ with $n$ punctures and/or marked points.  We define the {\em complexity} of $S$ as $\xi(S) =3g-3+n$. We say that a simple closed curve on $S$ is {\em essential} if it does not bound a disk or a once-punctured disk on $S$.  An {\em essential subsurface} of $S$ is a properly embedded subsurface $N \subset S$ for which each boundary component is an essential curve in $S$.

The {\em curve complex} $\C(S)$ of $S$ is a simplicial complex whose $k$-simplices correspond to sets of $k+1$ isotopy classes of essential simple closed curves on $S$ with pairwise disjoint representatives.   In order to simplify the notation, a set of isotopy classes of simple closed curves will be confused with its representative curves, the corresponding vertices of $\C(S)$, and the subcomplex of $\C(S)$ spanned by the vertices.  We also assume that representatives of isotopy classes of curves and subsurfaces intersect {\em minimally} (that is, transversely and in the minimal number of components), and denote by $i(\alpha,\beta)$ their intersection number.

If $\xi(S) > 1$, then $\C(S)$ is a connected complex of dimension $\xi(S)-1$. If $\xi(S)\le 0$ and $S\ne S_{1,0}$, then $\C(S)$ is empty. If $\xi(S)=1$ or $S=S_{1,0}$, then $\C(S)$ is a countable set of vertices; in order to obtain a connected complex, we modify the definition of $\C(S)$ by declaring $\alpha, \beta\in \C^{(0)}(S)$ to be adjacent in $\C(S)$ whenever $i(\alpha,\beta)=1$ if $S=S_{1,1}$ or $S=S_{1,0}$, and whenever $i(\alpha,\beta)=2$ if $S=S_{0,4}$.  Furthermore, we add triangles to make $\C(S)$ into a flag complex.   In all three cases, the complex $\C(S)$ so obtained is isomorphic to the well-known {\em Farey complex}.

We recall some definitions and results from \cite{AL} that we will need later.

\begin{definition} [Detectable intersection] Let $S$ be a surface and $Y\subset \C(S)$ a subcomplex.  If $\alpha,\beta \in Y$ are curves with $i(\alpha,\beta) \neq 0$, then we say that their intersection is {\em $Y$--detectable} (or simply {\em detectable} if $Y$ is understood) if there are two pants decompositions $P_\alpha,P_\beta \subset Y$ such that
\begin{equation} \label{E:detectable} \alpha \in P_\alpha, \, \beta \in P_\beta, \mbox{ and } P_\alpha - \alpha = P_\beta - \beta. \end{equation}
\label{D:detectable}
\end{definition}
We note that if $\alpha,\beta$ have detectable intersection, then they must fill a $\xi=1$ (essential) subsurface, which we denote $ N(\alpha\cup\beta) \subset S$.   For notational purposes, we call $P = P_\alpha - \alpha = P_\beta - \beta$ a {\em pants decomposition of $S - N(\alpha\cup\beta)$}, even though it includes the boundary components of $N(\alpha\cup\beta)$. The following is Lemma 2.3 in \cite{AL}:

\begin{lemma} \label{L:detectable2detectable}
Let $Y \subset \C(S)$ be a subcomplex, and $\alpha,\beta \in Y$ intersecting curves with $Y$--detectable intersection.  If $\phi:Y \to \C(S)$ is a locally injective simplicial map, then $\phi(\alpha),\phi(\beta)$ have $\phi(Y)$--detectable intersection, and hence fill a $\xi=1$ subsurface.
\end{lemma}

\subsection{Farey neighbors} A large part of our arguments will rely on being able to recognize when two curves are {\em Farey neighbors}, as we now define: 

\begin{definition}[Farey neighbors]
Let  $\alpha$ and $\beta$ be curves on $S$ which fill a $\xi=1$ subsurface $N \subset S$.  We say $\alpha$ and $\beta$ are {\em Farey neighbors} if they are adjacent in $\C(N)$. 
\end{definition}

The following result is a useful tool for recognizing Farey neighbors, and is a rephrasing of  Lemma 2.4 in \cite{AL} (see also the comment immediately after it):

\begin{lemma} \label{l:fareydetect} Suppose $\alpha_1,\alpha_2,\alpha_3,\alpha_4$ are  curves on $S$ such that:
\begin{enumerate}
\item $\alpha_2,\alpha_3$ together fill a $\xi=1$ subsurface $N\subset S$
\item $i(\alpha_i,\alpha_j) = 0 \Leftrightarrow |i-j| > 1$ for all $i \neq j$. 
\item $\alpha_1$ and $\alpha_4$ have non-zero intersection number with exactly one component of $\partial N$.
\end{enumerate}
 Then $\alpha_2,\alpha_3$ are Farey neighbors.
\end{lemma}

\section{Enlarging rigid sets}
\label{s:3}

In this section we discuss the problem of enlarging rigid sets of the curve complex.  We recall the definition of rigid set from \cite{AL}: 

\begin{definition}[Rigid set] Suppose $S\ne S_{1,2}$. 
We say that  $Y \subset \C(S)$ is {\em rigid} if for every locally injective simplicial map $\phi:Y \to \C(S)$ there exists $h\in \Mod^{\pm}(S)$ with $h|_{Y} = \phi$,  unique up to the pointwise stabilizer of $Y$ in $\Mod^{\pm}(S)$.
\end{definition}

\begin{remark} The definition above may seem somewhat different to the one used in \cite{AL}, where we used the group $\Aut(\C(S))$ instead of $\Mod^{\pm}(S)$. Nevertheless, in the light of the results of Ivanov \cite{Ivanov}, Korkmaz \cite{Korkmaz} and Luo \cite{Luo} mentioned in the introduction, the two definitions are essentially the same as $S\ne S_{1,2}$. For $S=S_{1,2}$, however, we will use the group $\Aut(\C(S))$ instead of $\Mod^{\pm}(S)$, due to the existence of {\em non-geometric} automorphisms of $\C(S)$. 
\end{remark}

The main step in the proof of Theorem \ref{main} is to enlarge the rigid sets constructed in \cite{AL} in a way that the sets we obtain remain rigid. As we mentioned in the introduction, while one might be tempted to guess that a set that contains a rigid set is necessarily rigid, this is far from true, as the next result shows: 

\begin{proposition}
Let $S=S_{0,n}$, with $n\ge 5$, and $\XX$ the finite rigid set identified in \cite{AL} (defined in Section \ref{S:punctured spheres}). For every curve $\alpha\in \C(S) \setminus \XX$, the set $\XX \cup \{\alpha\}$
is not rigid. 
\label{onecurve}
\end{proposition}

\begin{proof}
Let $S_{\alpha}$ be the smallest subsurface of $S$ containing those curves in $\XX$ which are disjoint from $\alpha$, and $S_\alpha'$ the connected component of $S\setminus S_\alpha$ that contains $\alpha$; from the construction in \cite{AL}, every component of $\partial S_\alpha'$ which is essential in $S$ is an element of $\XX$.  Consider a mapping class  $f\in \Mod(S)$ that is pseudo-Anosov on $S'_\alpha$ and the identity on $S_\alpha$. Define a map $\phi:\XX\cup \{\alpha\}\to \C(S)$ by 
$\phi(\beta) = \beta$ for all $\beta \ne \alpha$, and $\phi(\alpha) = f(\alpha)$. By construction, the map $\phi$ is locally injective and simplicial, but cannot be the restriction
of an element of $\Mod^{\pm}(S)$.
\end{proof}

While Proposition \ref{onecurve} serves to highlight the obstacles for enlarging a rigid set to a set that is also rigid, we now explain two procedures for doing so. First, we recall the following definition from \cite{AL}:

\begin{definition}
Let $A$ be a set of curves in $S$.
\begin{enumerate}
\item $A$ is {\em almost filling (in $S$)} if the set
\[ B = \{ \beta \in \C^{(0)}(S) \setminus A \mid i(\alpha,\beta) = 0 \, \forall \alpha \in A \} \]
is finite.  In this case, we call $B$ the {\em set of curves determined by $A$}.
\item If $A$ is almost filling (in $S$), and $B = \{ \beta \}$ is a single curve, then we say that $\beta$ is uniquely determined by $A$.
\end{enumerate}
\end{definition}

An immediate consequence of the definition is the following.

\begin{lemma}
Let $Y$ be a rigid set of curves, and $A \subset Y$ an almost filling in $S$. If $\beta$ is uniquely determined by $A$, then $Y \cup \{ \beta \}$ is rigid.
\label{determine}
\end{lemma}
\begin{proof}
Given any locally injective simplicial map $\phi \colon Y \cup \{ \beta \} \to \C(S)$, we let $f \in \Mod^{\pm}(\C(S))$ be such that $f|_Y = \phi$.  Then $f(\beta)$ is the unique curve determined by $f(A) = \phi(A)$.  On the other hand, $\phi(\beta)$ is connected by an edge to every vertex in $\phi(A)$, since $\phi$ is simplicial.  Since $\phi$ is injective on the star of $\beta$, it is injective on $\beta \cup A$, and so $\phi(\beta) \not \in A$.  It follows that $\phi(\beta)$ is the curve uniquely determined by $\phi(A)$, and hence $f(\beta) = \phi(\beta)$.
\end{proof}

In particular, this gives rise to one method for enlarging a  rigid set which we formalize as follows.
Given a subset $Y \subset \C(S)$ define
\[ Y' = Y \cup \{ \beta \mid \beta \mbox{ is uniquely determined by some almost filling set } A \subset Y \}. \]
From this we recursively define $Y^{r} = (Y^{r-1})'$ for all $r > 0$ where $Y = Y^{0}$. Observe that, as an immediate consequence of Lemma \ref{determine}, we obtain:
\begin{proposition}
If $Y \subset \C(S)$ is a rigid set, then so is $Y^r$ for all $r \geq 0$.
\label{p:prime}
\end{proposition}

Next, we give a sufficient condition for the union of two  rigid sets to be rigid. Before doing so, we need the following definition: 

\begin{definition}[Weakly rigid set]
We say that a set $Y\subset \C(S)$ is {\em weakly rigid} if , whenever $h,h'\in \Mod^{\pm}(S)$ satisfy $h|_Y = h'|_Y$, then $h = h'$.
\end{definition}

Alternatively, $Y$ is weakly rigid if the pointwise stabilizer in $\Mod^\pm(S)$ is trivial.  Note that if $Y$ is a weakly rigid set, then so is every set containing $Y$.

\begin{lemma}
Let $Y_1, Y_2 \subset \C(S)$ be rigid sets.  If $Y_1\cap Y_2$ is weakly rigid then $Y_1\cup Y_2$ is rigid. 
\label{L:glue}
\end{lemma}

\begin{proof}
Let $\phi:Y_1\cup Y_2\to \C(S)$ be a locally injective simplicial map. Since $Y_i$ is rigid and has trivial pointwise stabilizer in $\Mod^{\pm}(S)$ (because $Y_1 \cap Y_2$ does), there exists a unique
$h_i\in\Mod^{\pm}(S)$ such that $h_i|_{Y_i} = \phi|_{Y_i}$. Finally, since $Y_1\cap Y_2$ is weakly rigid we have $h_1 = h_2=h$. Therefore $h|_{Y_1\cup Y_2} = \phi$, and the result follows. 
\end{proof}

We now proceed to describe our second method for enlarging a rigid set. We start with some definitions and notation. We write $T_\alpha$ for the Dehn  twist along a curve $\alpha$. Recall that the half twist $H_\alpha$ about a curve $\alpha$ is defined if and only if the curve cuts off a pair of pants containing two punctures of $S$.  Furthermore, there is exactly one half twist about $\alpha$ if in addition $S$ is not a four-holed sphere.
\begin{definition}
Say that Farey neighbors $\alpha,\beta$ are {\em twistable} if either:
\begin{itemize}
\item $N(\alpha \cup \beta)$ is a one-holed torus or
\item $N(\alpha \cup \beta)$ is a four-holed sphere and $H_\alpha,H_\beta$ are both defined and unique.
\end{itemize}
In this situation we define $f_\alpha = T_\alpha$ and $f_\beta= T_\beta$ in the first case and $f_\alpha = H_\alpha$ and $f_\beta = H_\beta$ in the second.  We call $f_\alpha,f_\beta$ the {\em twisting pair} for $\alpha,\beta$.

In case (1) we call $\alpha,\beta$ {\em toroidal} and in case (2) we call them {\em spherical}.
\end{definition}
We note that whether twistable Farey neighbors $\alpha,\beta$ are toroidal or spherical can be distinguished (i) by $i(\alpha,\beta)$ (whether it is $1$ or $2$), (ii) by the homeomorphism types of $\alpha$ and $\beta$ (whether they are  nonseparating curves or they cut off a pair of pants), or (iii) by the homeomorphism type of $N(\alpha \cup \beta)$ (whether it is a one-holed torus or a four-holed sphere).

The following well-known fact describes the common feature of these two situations.
\begin{proposition} \label{P:twistable_Farey_neighbors}
Suppose $\alpha,\beta$ are twistable Farey neighbors and that $f_\alpha,f_\beta$ is their twisting pair.  Then
\[ f_\alpha(\beta) = f_\beta^{-1}(\alpha) \mbox{ and } f_\alpha^{-1}(\beta) = f_\beta(\alpha), \]
and these are the unique common Farey neighbors of both $\alpha$ and $\beta$.
\end{proposition}

Sets of twistable Farey neighbors which interact with each other frequently occur in rigid sets.  We distinguish one particular type of such set in the following definition:
\begin{definition} \label{D:strings}
Suppose $Y$ is a rigid subset of $\C(S)$ and $ A = \{ \alpha_1,\ldots,\alpha_k \} \subset Y$.  
We say that $A$ is a {\em closed string of Farey neighbors} in $Y$ provided the following conditions are satisfied, counting indices modulo $k$:
\begin{enumerate}
\item The curves $\alpha_i,\alpha_{i+1}$ are twistable Farey neighbors with twisting pair $f_{\alpha_i},f_{\alpha_{i+1}}$.
\item  $i(\alpha_i,\alpha_{i+1}) \neq 0$ is $Y$--detectable.
\item   $i(\alpha_i,\alpha_j) = 0$ if $i-j \neq \pm 1$ modulo $k$.
\item $\alpha_i,\alpha_{i+1},\alpha_{i+2},\alpha_{i+3}$ satisfies the hypothesis of Lemma~\ref{l:fareydetect}.
\end{enumerate}

Given a closed string of twistable Farey neighbors $ A \subset Y$, we define $$Y_A= Y \cup \{ f_{\alpha_i}^{\pm 1}(\alpha_j) \}_{i,j = 1}^k.$$
\end{definition}

\begin{remark} Two comments are in order:
\begin{enumerate}
\item There is a priori some ambiguity in the notation as $f_{\alpha_i}$ can be defined as part of the twisting pair for $\alpha_i,\alpha_{i+1}$ as well as for $\alpha_{i-1},\alpha_i$.  However, if $\alpha_i$ is part of two pairs of different twistable Farey neighbors in $Y$, then they must both be toroidal or both spherical as this is determined by the homeomorphism type of $\alpha_i$.  Consequently, the mapping class $f_{\alpha_i}$ is independent of what twistable pair it is included in.
\item The set $Y_A$ has a more descriptive definition, given condition (3) of Definition~\ref{D:strings}.  Namely,
\[ Y_A = Y \cup \{ f_{\alpha_i}^{\pm 1}(\alpha_j) \mid i - j = \pm 1 \mbox{ modulo } k \}. \]
\end{enumerate}
\end{remark}

See Figure \ref{F:twistable} for an example of a closed string of twistable Farey neighbors and two of their images under the twisting pair.

\begin{figure}[htb]
\begin{center}
\includegraphics[width=2in,height=2in]{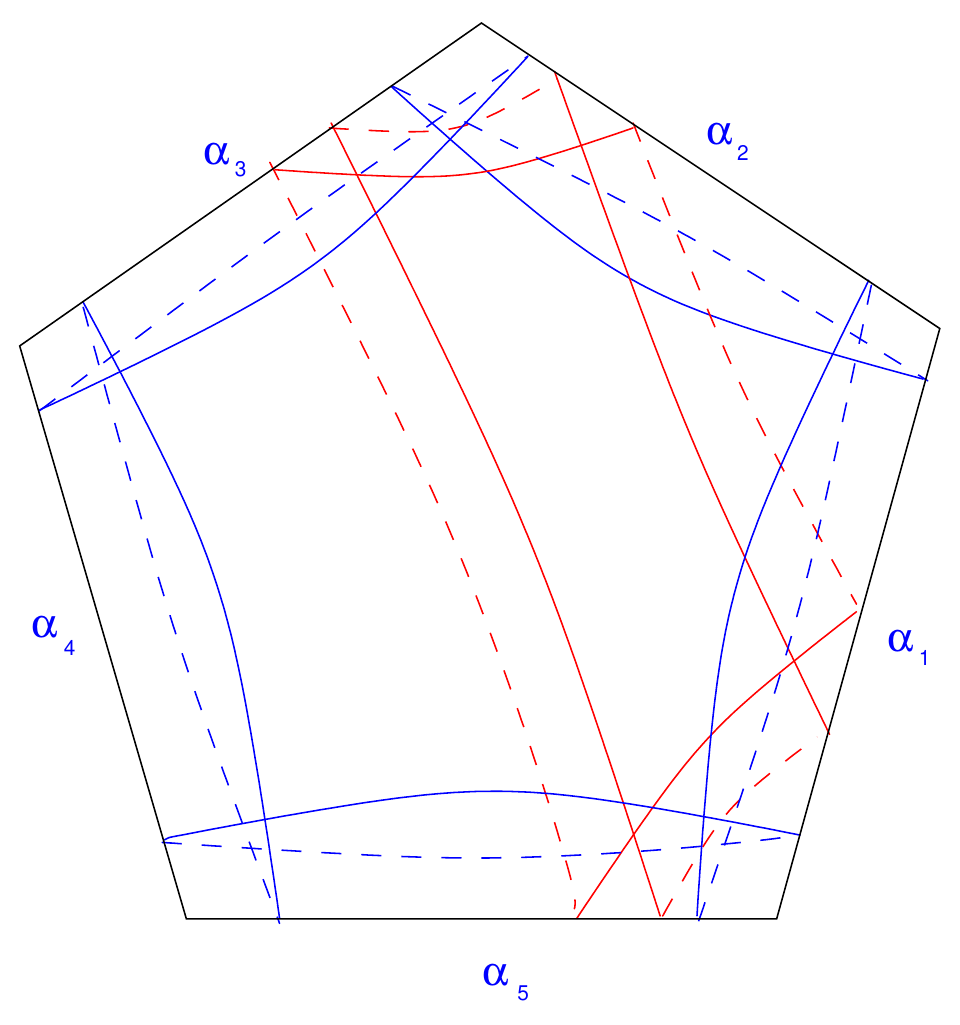} \caption{The set $Y= A = \{\alpha_1, \ldots, \alpha_5\}$ is the rigid set $\XX(S_{0,5})$ identified in \cite{Luo, AL}, and is a closed string of twistable Farey neighbors. The red curves in the picture are $f_{\alpha_1}(\alpha_2), f_{\alpha_2}(\alpha_1)$,  for the twistable pair $\alpha_1,\alpha_2$.  The automorphism group of $Y_A$ that fixes $Y$ pointwise is generated by an orientation-reversing involution $\sigma:S_{0,5} \to S_{0,5}$ that fixes $\alpha_i$ and interchanges $f_{\alpha_i}$ and $f_{\alpha_{i+1}}$, for all $i$ (mod 5).  } \label{F:twistable}
\end{center}
\end{figure}

The situation in the next proposition arises in multiple settings, and provides a way to extend a rigid set to a larger set which is nearly rigid.
\begin{proposition} \label{P:strings_prop1}
Let $Y$ be a rigid subset of $\C(S)$ and $A = \{ \alpha_1,\ldots,\alpha_k\} \subset Y$  a closed string of twistable Farey neighbors in $Y$.  Then, counting indices modulo $k$:
\begin{enumerate}
\item $f_{\alpha_i}^{\pm 1}(\alpha_{i+1}) = f_{\alpha_{i+1}}^{\mp 1}(\alpha_i)$ are the unique common Farey neighbors of $\alpha_i$ and $\alpha_{i+1}$, 
 \item  $i(f_{\alpha_i}^{\pm 1}(\alpha_j),\alpha_{j'}) \neq 0$ for all $i$ and all $(j,j') \in \{(i+1,i),(i+1,i+1),(i-1,i),(i-1,i-1)\}$.Furthermore, these intersections are $Y_A$--detectable.
\item For any locally injective simplicial map $\phi \colon Y_A \to \C(S)$, 
\[ \phi(f_{\alpha_i}(\alpha_{i+1})) = \phi(f_{\alpha_{i+1}}^{-1}(\alpha_i))\mbox{ and } \phi(f_{\alpha_i}^{-1}(\alpha_{i+1}))  = \phi(f_{\alpha_{i+1}}(\alpha_i)) \] are the unique Farey neighbors of $\phi(\alpha_i)$ and $\phi(\alpha_{i+1})$.
\end{enumerate}
\label{l:twistable}
\end{proposition}

\begin{proof}
Conclusion (1) follows immediately from Definition~\ref{D:strings} part (1) and Proposition~\ref{P:twistable_Farey_neighbors}.

Next we prove conclusion (2).   Fix $(j,j')$ as in the proposition.  Then since $i(\alpha_i,\alpha_j) \neq 0$, it follows that $f_{\alpha_i}(\alpha_j)$ nontrivially intersects both $\alpha_i$ and $\alpha_j$.  Since $\alpha_{j'}$ is one of these latter two curves, the first statement follows.  By part (2) of Definition~\ref{D:strings}, $i(\alpha_i,\alpha_j) \neq 0$ is $Y$--detectable.  Let $P_{\alpha_i},P_{\alpha_j} \subset Y$ be pants decompositions containing $\alpha_i$ and $\alpha_j$, respectively, as in Definition~\ref{D:detectable}, and set $P = P_{\alpha_i} - \alpha_i = P_{\alpha_j} - \alpha_j$.  Then since $f_{\alpha_i}$ is  supported in $N(\alpha_i \cup \alpha_j)$ which is contained in the complement of $P$, we can define two more pants decompositions
\[ P_{f_{\alpha_i}^{\pm 1}(\alpha_j)} = P \cup f_{\alpha_i}^{\pm 1}(\alpha_j) \subset Y_A. \]
Together with $P_{\alpha_i}$ and $P_{\alpha_j}$ these are sufficient to detect all the intersections claimed.  In all cases, $P \subset Y$ is the pants decomposition of the complement of $N(\alpha_i \cup \alpha_j)$, as required.

For conclusion (3), we explain why $\phi(f_{\alpha_i}(\alpha_{i+1}))$ and $\phi(\alpha_i)$ are Farey neighbors.  The other three cases are similar.  For this, we consider the set
\[\{ \, \,  \phi(f_{\alpha_i}(\alpha_{i-1})) \, \, , \, \, \phi(\alpha_i) = \phi(f_{\alpha_i}(\alpha_i))  \, \, , \, \, \phi(f_{\alpha_i}(\alpha_{i+1})) \, \, , \, \, \phi(\alpha_{i+2}) =\phi(f_{\alpha_i}(\alpha_{i+2})) \, \, \}. \]
The equalities here follow from the disjointness property (3) of Definition~\ref{D:strings} since a Dehn twist or half-twist has no effect on a curve that is disjoint from the curve supporting the twist.  The goal is to prove that all three conditions of Lemma~\ref{l:fareydetect} are satisfied.

By part (2) of the proposition and Lemma~\ref{L:detectable2detectable} it follows that any two consecutive curves in this set have $\phi(Y_A)$--detectable intersections, and fill a $\xi$=1 subsurface.  Therefore condition (1) of Lemma~\ref{l:fareydetect} is satisfied for this set of curves.  Since $\alpha_{i-1},\alpha_i,\alpha_{i+1},\alpha_{i+2}$ satisfy condition (2) of Lemma~\ref{l:fareydetect} and the given set is the image of these under the simplicial map $\phi \circ f_{\alpha_i}$, these curves also satisfy condition (2) of Lemma~\ref{l:fareydetect}.

Finally, we wish to verify that condition (3) of Lemma~\ref{l:fareydetect} is satisfied.  Since $Y$ is rigid, there exists $f \in \Mod^{\pm 1}(S)$ inducing $\phi|_{Y}$.  We also note that $N = N(\alpha_i \cup \alpha_{i+1}) = N(\alpha_i \cup f_{\alpha_i}(\alpha_{i+1}))$ has only one boundary component---all other holes of this subsurface (if any) must be punctures of $S$.  
Since the pants decomposition of the complement of $N$ is
\[ P = P_{\alpha_i} - \alpha_i = P_{\alpha_{i+1}} - \alpha_{i+1} = P_{f_{\alpha_i}(\alpha_{i+1})} - f_{\alpha_i}(\alpha_{i+1}).\]
and is contained in $Y$, we see that $\phi(P) = f(P)$.  Since this is used in the $\phi(Y_A)$--detection of both $i(\phi(\alpha_i),\phi(\alpha_{i+1})) \neq 0$ and $i(\phi(\alpha_i),\phi(f_{\alpha_i}(\alpha_{i+1}))) \neq 0$, we have
\[ f(N) = N(f(\alpha_i) \cup f(\alpha_{i+1})) = N(\phi(\alpha_i) \cup \phi(\alpha_{i+1})) = N(\phi(\alpha_i) \cup \phi(f_{\alpha_i}(\alpha_{i+1}))).\]
Consequently, this surface has only one boundary component, and so condition (3) of Lemma~\ref{l:fareydetect} is satisfied.
\end{proof}

\begin{proposition} \label{P:closed_string}
If $Y$ is a rigid subset of $\C(S)$ and $A = \{ \alpha_1,\ldots,\alpha_k \}$ is a closed string of twistable Farey neighbors in $Y$, then any locally injective simplicial map $\phi \colon Y_A \to \C(S)$ which is the identity on $Y$ satisfies $\phi(Y_A) = Y_A$.  Furthermore, the subgroup of the automorphism group of $Y_A$ fixing $Y$ pointwise has order at most $2$.  If this subgroup is nontrivial, then it is generated by the involution $\sigma \colon Y_A \to Y_A$ given by $\sigma(f_{\alpha_i}(\alpha_j)) = f_{\alpha_i}^{-1}(\alpha_j)$ for all $i,j$ (or equivalently, for all $i,j$ with $i-j = \pm 1$ (modulo $k$)).
\end{proposition}
\begin{proof}
Since $f_{\alpha_i}^{\pm 1}(\alpha_{i+1})$ is the unique pair of common Farey neighbors of $\alpha_i,\alpha_{i+1}$, and since $\phi(\alpha_i) = \alpha_i$, $\phi(\alpha_{i+1}) = \alpha_{i+1}$, Proposition~\ref{P:strings_prop1} implies that for every $i$ and $j$ with $i-j = \pm 1$ (modulo $k$), we have
\[  \{ \phi(f_{\alpha_i}(\alpha_j)),\phi(f_{\alpha_i}^{-1}(\alpha_j)) \} = \{ f_{\alpha_i}(\alpha_j),f_{\alpha_i}^{-1}(\alpha_j) \}, \]
and so the first claim of the proposition follows.

Next we suppose $\phi$ is any automorphism of $Y$ that restricts to the identity on $Y$.  We claim that if there is some $i,j$ with $i-j = \pm 1$ (modulo $k$) so that $\phi(f_{\alpha_i}(\alpha_j)) = f_{\alpha_i}^{-1}(\alpha_j)$, then this is true for every $i,j$ with $i-j = \pm 1$ (modulo $k$).  To this end, suppose that $\phi(f_{\alpha_i}(\alpha_{i+1})) = f_{\alpha_i}^{-1}(\alpha_{i+1})$ for some index $i$ (the case $\phi(f_{\alpha_i}(\alpha_{i-1})) = f_{\alpha_i}^{-1}(\alpha_{i-1})$ is similar).  Then note that
\[ i(f_{\alpha_i}(\alpha_{i-1}),f_{\alpha_i}(\alpha_{i+1})) = 0 = i(f_{\alpha_i}^{-1}(\alpha_{i-1}),f_{\alpha_i}^{-1}(\alpha_{i+1})) \]
while
\[ i(f_{\alpha_i}(\alpha_{i-1}),f_{\alpha_i}^{-1}(\alpha_{i+1})) \neq 0 \neq i(f_{\alpha_i}^{-1}(\alpha_{i-1}),f_{\alpha_i}(\alpha_{i+1})). \]
Since $\phi$ is simplicial and locally injective, we must have $\phi(f_{\alpha_i}(\alpha_{i-1})) = f_{\alpha_i}^{-1}(\alpha_{i-1})$ and $\phi(f_{\alpha_i}^{-1}(\alpha_{i-1})) = f_{\alpha_i}(\alpha_{i-1})$.  Consequently,
\[ \phi(f_{\alpha_{i-1}}(\alpha_i)) = \phi(f_{\alpha_i}^{-1}(\alpha_{i-1})) = f_{\alpha_i}(\alpha_{i-1}) = f_{\alpha_{i-1}}^{-1}(\alpha_i).\]
Repeating this argument again, it follows that for all $i$, $\phi(f_{\alpha_i}(\alpha_{i+1})) = f_{\alpha_i}^{-1}(\alpha_{i+1})$, as required.  Thus, in this case, $\phi$ is given by $\sigma$ as in the statement of the proposition.

If we are not in the situation of the previous paragraph, then it follows that $\phi$ is the identity, completing the proof.
\end{proof}

After this discussion we are  in a position to explain how to obtain an exhaustion of $\C(S)$ by finite rigid sets. Here, $\Mod(S)$ denotes the index 2 subgroup of $\Mod^{\pm}(S)$ consisting of those mapping classes that preserve orientation. 

\begin{proposition}
Let $Y\subset  \C(S)$ be a finite rigid set such that  $\Mod(S)\cdot Y = \C(S)$. Suppose there exists $G\subset Y$ such that:

\begin{enumerate}
\item The set $\{f_\alpha \mid \alpha \in G\}$ generates $\Mod(S)$;
\item  $Y \cap f_\alpha(Y)$ is weakly rigid, for all $\alpha \in  G$.
\end{enumerate}
Then there exists an sequence $Y= Y_1 \subset Y_2 \subset \ldots \subset Y_n \subset \ldots$ such that $Y_i$ is a finite rigid set, has trivial pointwise stabilizer in $\Mod^{\pm}(S)$ for all $i$, and $$\bigcup_{i\in \mathbb N} Y_i = \C(S).$$ 
\label{P:finish}
\end{proposition}

\begin{proof}
First, the fact that $Y$ is rigid implies that  $f_\alpha(Y)$ is rigid for all $\alpha \in Y$. Therefore,  the set $Y_2:=Y \cup f_G(Y)$ is also rigid by assumption (2) and repeated application of Lemma~\ref{L:glue}. We now define, for all $n\ge 2$, $$Y_{n+1}:= Y_n \cup f_G (Y_n).$$ By induction, we see that $Y_n$ is rigid for all $n$ and so the first claim follows. 
Next, the pointwise stabilizer of $Y$ in $\Mod^{\pm}(S)$ is trivial because $Y \cap f_\alpha(Y)$ is weakly rigid. Therefore, $Y_n$ has  trivial pointwise stabilizer in $\Mod^{\pm}(S)$,
as $Y\subset Y_n$ for all $n$. 
Finally, since $\{f_\alpha \mid \alpha \in  G\}$ generates $\Mod(S)$ and $\Mod(S)\cdot Y = \C(S)$, it follows that  $$\bigcup_{i\in \mathbb N} Y_i = \C(S),$$ 
which completes the proof.
\end{proof}

We end this section by explaining how Theorem \ref{main} implies that curve complexes are simplicially rigid: 

\begin{proof}[Proof of Corollary \ref{C:Ivanov}]
Let $S\ne S_{1,2}$, and let $\phi:\C(S)\to \C(S)$ be a locally injective simplicial map. Let $\XX_1\subset \XX_2 \subset \ldots $ be the exhaustion of $\C(S)$ provided by Theorem \ref{main}. Since $\XX_i$ is rigid and has trivial pointwise stabilizer in $\Mod^{\pm}(S)$, there exists a unique mapping class $h_i\in \Mod^{\pm}(S)$ such that $h_i|_{\XX_i} = \phi|_{\XX_i}$. Finally, Lemma \ref{L:glue} implies that $h_i= h_j$ for all $i,j$, and thus the result follows. 
\end{proof}

\section{Punctured spheres} \label{S:punctured spheres}

In this section we prove Theorem \ref{main} for $S=S_{0,n}$. If $n\le 3$ then $\C(S)$ is empty and thus the result is trivially true. The case $n=4$ is dealt with at the end of this section, as it needs special treatment. Thus, from now on we assume that $n\ge 5$. As in \cite{AL} we represent $S$ as the double of an $n$-gon $\Delta$ with vertices removed, and define $\XX$ as the set of curves on $S$ obtained by connecting every non-adjacent pair of sides of $\Delta$ by a straight line segment and then doubling; see Figure \ref{F:octagonarcs} for the case $n = 8$:

\begin{figure}[htb]
\begin{center}
\includegraphics[width=1.5in,height=1.5in]{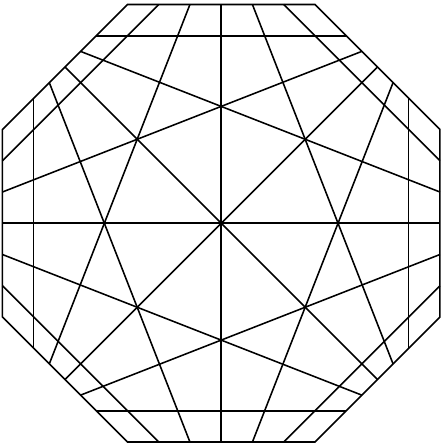} \caption{Octagon and arcs for $S_{0,8}$.} \label{F:octagonarcs}
\end{center}
\end{figure}

Note that the point-wise stabilizer of $\XX$ in $\Mod^{\pm}(S)$ has order two, and is generated by an orientation-reversing involution $i:S\to S$ that interchanges the two copies of  $\Delta$. The rigidity of the set $\XX$, which was established in \cite{AL}, may be rephrased as follows:

\begin{theorem}[\cite{AL}]\label{T:sphere}
For any locally injective simplicial map $\phi:\XX \to \C(S)$, there exists a unique $h \in \Mod(S)$ such that $h|_\XX = \phi$, unique up to precomposing with  $i$. \label{t:sphere}
\end{theorem}

We are going to enlarge the set $\XX$ in the fashion described in Section \ref{s:3}. We number the sides of the $\Delta$ in a cyclic order, and denote by $\alpha_j$ the curve defined by the arc on $\Delta$ that connects the sides with labels $j$ and $j+2$ mod $n$. Let $A= \{\alpha_1, \ldots, \alpha_n\}$; in the terminology of \cite{AL}, $A$ is the set of {\em chain curves} of $\XX$. 
Observe that every element of $A$ bounds a disk containing  exactly two punctures of $S$, and that if two elements of $A$ have non-zero intersection number then they are Farey neighbors in $\XX$. Thus we see that $A$ is a closed string of $n$ twistable Farey neighbors, and may consider the set $\XX_A$ from Definition \ref{D:strings}. As a first step towards proving Theorem \ref{main} for $S_{0,n}$, we show that $\XX_A$ is rigid. Since the pointwise stabilizer of $\XX_A$ is trivial, this amounts to the following statement: 

\begin{theorem}
\label{T:sphere-larger}
For any locally injective simplicial map $\phi:\XX_A \to \C(S)$, there exists a unique $g \in \Mod^\pm(S)$ such that $g|_{\XX_A} = \phi$.
\end{theorem}

\begin{proof}
Let $\phi:\XX_A \to \C(S)$ be a locally injective simplicial map. By Theorem \ref{t:sphere}, there exists  $h \in \Mod^{\pm}(S)$ such that $h|_{\XX} = \phi|_{\XX}$, unique up to precomposing with the involution $i$. Since  $i$ fixes every element of $\XX$, after precomposing $\phi$ with $h^{-1}$ we may assume that $\phi|_{\XX}$ is the identity map. By Proposition \ref{P:closed_string}, $\phi(\XX_A) =\XX_A$; moreover, the automorphism group of $\XX_A$ fixing
$\XX$ pointwise has order two, generated by the involution $\sigma: \XX_A \to \XX_A$ that interchanges $f_{\alpha_i}(\alpha_{i+1}) $ and $f_{\alpha_{i+1}}^{-1}(\alpha_i)$ for all $i$. Since $ i|_{\XX_A}=\sigma$, up to precomposing $\phi$ with $i$, we deduce that $\phi|_{\XX_A}$ is the identity, as we wanted to prove.
\end{proof}

We now  prove Theorem \ref{main} for spheres with punctures: 

\begin{proof}[Proof of Theorem \ref{main} for $S=S_{0,n}$, $n\ge 5$]
Let $\XX_A$ be the set constructed above, which is rigid and has trivial pointwise stabilizer in $\Mod^{\pm}(S)$, by Theorem \ref{T:sphere-larger}. The set $\{H_\alpha \mid \alpha \in A\}$ generates $\Mod(S)$; see, for instance,  Corollary 4.15 of \cite{Farb-Margalit}. In addition,  $\XX_A \cap H_\alpha(\XX_A)$ is weakly rigid, for all $\alpha \in A$, as it contains $A$ and $H_{\alpha_i}(\alpha_j)$ for any $\alpha_i,\alpha_j$ disjoint from $\alpha$.  Finally, by inspection we see $\Mod(S)\cdot \XX_A = \C(S)$. Therefore, we may apply Proposition \ref{P:finish} to the sets $Y=\XX_A$ and $G= A$ to obtain the desired sequence  $\XX_A= Y_1 \subset Y_2 \subset \ldots \subset Y_n \subset \ldots$ of finite rigid sets.
\end{proof}

\subsection{Proof of Theorem \ref{main} for $S=S_{0,4}$} As mentioned in the introduction, in this case $\C(S)$ is isomorphic to the Farey complex. It is easy to see, and is otherwise explicitly stated in \cite{AL}, that any triangle in $\C(S)$ is rigid. From this, plus the fact that any edge in $\C(S)$ is contained in exactly two triangles, it follows that any subcomplex of $\C(S)$ that is homeomorphic to a disk is also rigid. Consider the dual graph of $\C(S)$ (which is in fact a trivalent tree $T$), equipped with the natural path-metric. Let $Y_1$ be a triangle in $\C(S)$, and define $Y_n$ to be the union of all triangles of $\C(S)$ whose corresponding vertices in $T$ are at distance at most $n$ from the vertex corresponding to $Y_1$. Then the sequence $(Y_n)_{n\in \mathbb N}$ gives the desired exhaustion of $\C(S)$.

\section{Closed and punctured surfaces of genus $g \ge 2$}

In this section we consider the case of a surface $S$ of genus $g\ge 2$ with $n\ge 0$ marked points.  First observe that if $g = 2$ and $n = 0$, then since $\C(S_{2,0}) \cong \C(S_{0,6})$ \cite{Luo}, the main theorem for $S_{2,0}$ follows from the case $S_{0,6}$, already proved in Section~\ref{S:punctured spheres}.   We therefore assume that $n \geq 1$ if $g = 2$.

We let $\XX \subset \C(S)$ denote the finite rigid set constructed in \cite{AL}. The definition of the set $\XX$ is somewhat involved and we will not recall it in full detail.  Instead, we first note that $\XX$ contains the set of {\em chain curves}
\[ \CC = \{\alpha_0^0,\ldots,\alpha_0^n,\alpha_1,\ldots,\alpha_{2g-1} \} \]
depicted in Figure  \ref{F:genus4chain}.
For notational purposes we also write $\alpha_0 = \alpha_0^1$ (and in case $n = 0$, $\alpha_0 = \alpha_0^0$).
In addition to these curves, $\XX$ contains every curve which occurs as the boundary component of a subsurface of $S$ filled by a subset $A \subset \CC$, provided its union is connected in $S$ and has one of the following forms:
\begin{enumerate}
\item $A = \{ \alpha_0^i,\alpha_0^j,\alpha_k \}$ where $0 \leq i \leq j \leq n$ and $k = 1$ or $2g+1$.
\item $A = \{ \alpha_0^i, \alpha_0^j, \alpha_k,\alpha_{k+1} \}$ where $0 \leq i \leq j \leq n$ and $k = 1$ or $2g$.
\item $A = \{ \alpha_i \mid i \in I  \}$ where $I \subset \{0,\ldots,2g+1\}$ is an interval (modulo $2g+2$).  If $n >0$ and $A$ has an odd number of curves, then we additionally require that the first and last numbers in the interval $I$ to be even.
\end{enumerate}
See Figure \ref{F:examples} for some key examples.

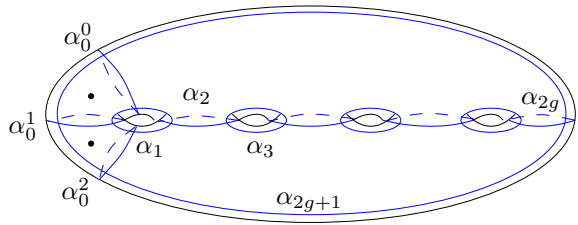
\begin{figure}[htb]
\labellist
\small\hair 2pt
 \pinlabel {$\alpha_{2g+1}$} [ ] at 125 10
 \pinlabel {$\alpha_{2g}$} [ ] at 234 57
 \pinlabel {$\alpha_1$} [ ] at 50 35
 \pinlabel {$\alpha_2$} [ ] at 72 60
 \pinlabel {$\alpha_3$} [ ] at 102 35
 \pinlabel {$\alpha_0^2$} [ ] at 15 15
 \pinlabel {$\alpha_0^1$} [ ] at -10 46
 \pinlabel {$\alpha_0^0$} [ ] at 16 88
\endlabellist
\centering
\includegraphics[scale=.8]{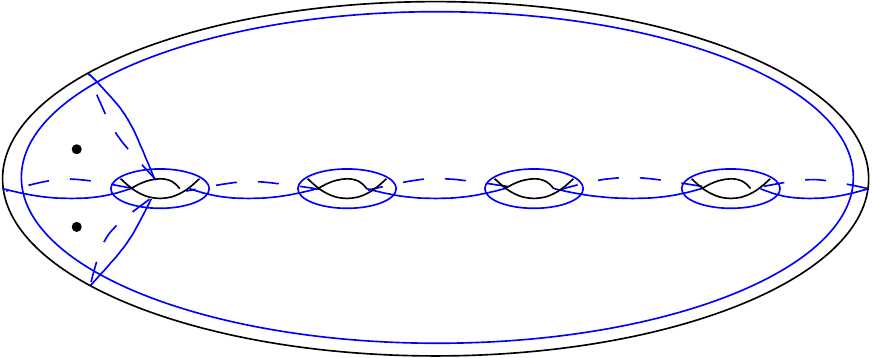}
\caption{Chain curves $\CC$ on a genus $4$ surface with $2$ marked points.} \label{F:genus4chain}
\end{figure}

\begin{figure}[htb]
\labellist
\small\hair 2pt
 \pinlabel {$\beta$} [ ] at 430 25
 \pinlabel {$\epsilon^{1\, 2}$} [ ] at 267 20
 \pinlabel {$\sigma^{2 \, 2}$} [ ] at 85 75
\endlabellist
\begin{center}
\includegraphics[width=5in]{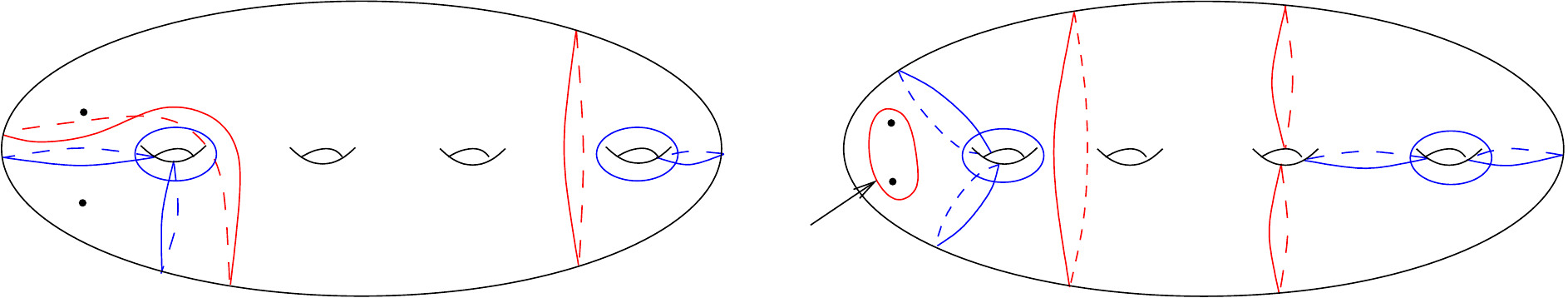} \caption{Examples of subsets of $\CC$ (in blue), together with the boundary components (in red) of the subsurface filled by them.  The red curves are in $\XX$.} \label{F:examples}
\end{center}
\end{figure}

The pointwise stabilizer of $\XX$ in $\Mod^{\pm}(S)$ is trivial. Thus the rigidity of the set $\XX$, established in \cite{AL}, may be rephrased as follows: 

\begin{theorem}[\cite{AL}]
Let $S = S_{g,n}$ with $g \geq 2$ and $n \geq 0$ (and $n \geq 1$ if $g = 2$).  For any locally injective simplicial map $\phi:\XX \to \C(S)$, there exists a unique $h \in \Mod^{\pm}(S)$ such that $h|_{\XX} = \phi$.
\label{t:rigidhigher}
\end{theorem}

It will be necessary to refer to some of the curves in $\XX$ by name, so we describe the naming convention briefly in those cases, along the lines of \cite{AL}.  We have already described the names of the elements of $\CC$.  For $0 < i < j  \leq n$ we let $\epsilon^{ij}$ be the boundary component of the subsurface $N(\alpha_1 \cup \alpha_0^{i-1} \cup \alpha_0^j)$ that also bounds a $(j-i+1)$--punctured disk in $S$ (containing the $i^{th}$ through $j^{th}$ punctures).  We call the curves $\epsilon^{ij}$ {\em outer curves}; see Figure~\ref{F:examples}.  For $0 < i \leq j \leq n$, we also consider the other boundary component of $N(\alpha_1 \cup \alpha_0^{i-1} \cup \alpha_j)$; this is a separating curve dividing the surface into two (punctured) subsurfaces of genus $1$ and $g-1$ respectively.  We denote this curve $\sigma^{ij}$.  One more curve in $\XX$ that we refer to as $\beta$ is shown in Figure~\ref{F:examples}, and is a component of the boundary of the subsurface $N(\alpha_{2g-2} \cup \alpha_{2g-1} \cup \alpha_{2g})$.

The strategy for proving Theorem \ref{main} for surfaces of genus $g\ge 2$ is similar in spirit to the one for punctured spheres, although considerably more involved. The main idea is to produce successive rigid enlargements of the rigid set $\XX$ identified in \cite{AL}, until we are in a position to apply Proposition \ref{P:finish}. 
We begin by replacing $\XX$ with $\XX'$, which is rigid by Proposition \ref{p:prime}.  For every $0 < j \leq n$, let 
\[ A_j = \{ \sigma^{i j} \mid 0 < i \leq j \} \cup \{ \sigma^{j i} \mid j \leq i \leq n \} \cup \{\alpha_1, \alpha_3, \alpha_4, \alpha_5, \ldots, \alpha_{2g+1} \}.\]
The set $A_j$ is almost filling and uniquely determines a curve denoted $\alpha_1^j$; see Figure~\ref{F:alpha1j}.  The naming is suggestive, as all $\alpha_1^j$ are homotopic to $\alpha_1$ upon filling in the punctures.

\begin{figure}[htb]
\labellist
\small\hair 2pt
 \pinlabel {$\alpha_1^2$} [ ] at 390 75
\endlabellist
\begin{center}
\includegraphics[width=5in]{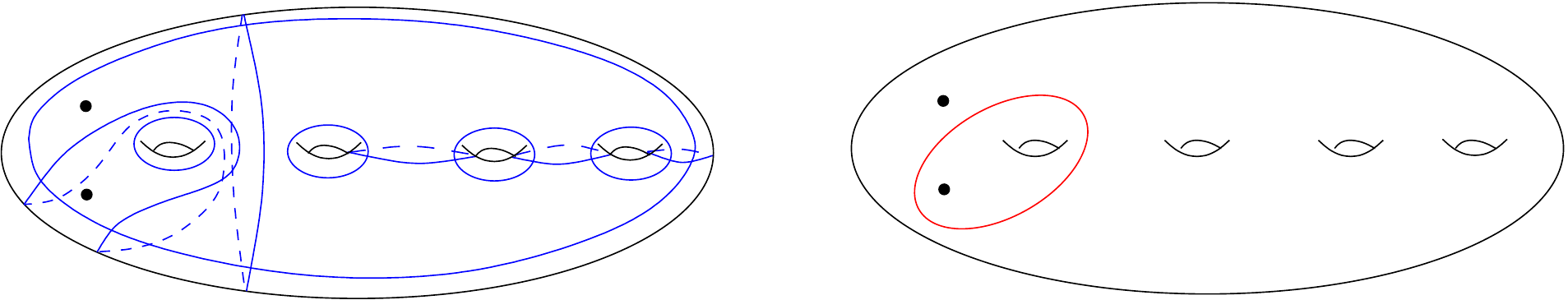} \caption{The surface on the left contains the set $A_2$ which uniquely determines $\alpha_1^2$.} \label{F:alpha1j}
\end{center}
\end{figure}

We can similarly find a subset $A_0$ (shown in the left of Figure~\ref{F:alpha10}) which is almost filling and uniquely determines a curve denoted $\alpha_1^0$ (shown on the right of Figure~\ref{F:alpha10}), which bounds a disk enclosing every puncture of $S$.  Consequently, $\alpha_1^j \in \XX'$, for all $j = 0,\ldots,n$, 

\begin{figure}[htb]
\labellist
\small\hair 2pt
 \pinlabel {$\alpha_1^0$} [ ] at 390 75
\endlabellist
\begin{center}
\includegraphics[width=5in]{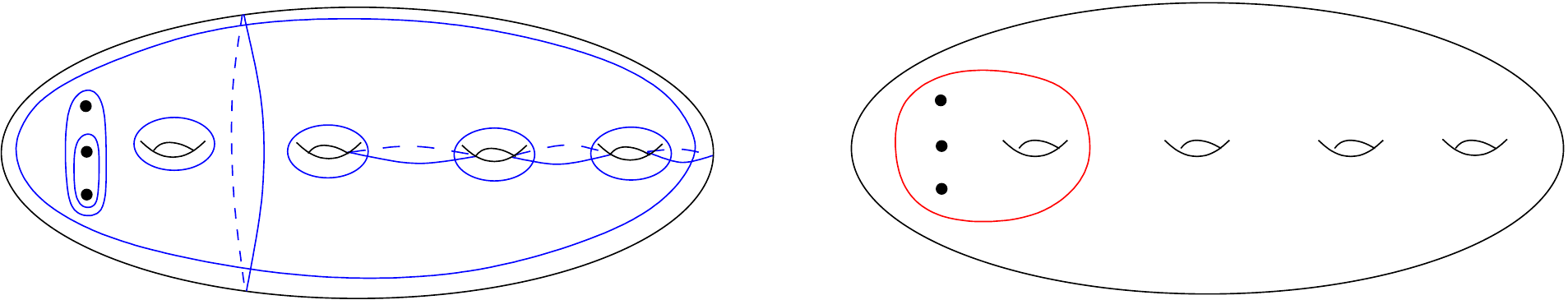} \caption{The curves $A_0 \subset \XX$ (left) and the curve $\alpha_1^0 \in \XX'$ (right).} \label{F:alpha10}
\end{center}
\end{figure}

\subsection{Punctured surface promotion}

One issue that arises only in the case $n > 0$ is that for intervals $I \subset \{0,\ldots,2g+1 \}$ (modulo $2g+2$) of odd length, the boundary curves of the neighborhood of the subsurface filled by $A= \{ \alpha_i \mid i \in I\}$ are only contained in $\XX$ when $I$ starts and ends with even indexed curves.  Passing to the set $\XX'$ allows us to easily enlarge further to a set which rectifies this problem. 

Specifically, we define $\XX_1$ to be the union of $\XX'$ together with boundary components of subsurfaces filled by sets $A = \{ \alpha_i,\alpha_{i+1},\ldots,\alpha_j \}$ where $0 < i \leq j \leq 2g-1$ and $i,j$ are both odd.  See Figure \ref{F:Aodd} for examples. 
Let $\BB_o$ be the set of all curves defined by such sets $A$.  

Before we proceed, we describe this set in more detail. 
Cutting $S$ open along $\alpha_1 \cup \alpha_3 \cup \ldots \cup \alpha_{2g-1} \cup \alpha_{2g+1}$ we obtain two components $\Theta_o^+$ and $\Theta_o^-$.  These are each spheres with holes: $\Theta_o^+$ is the sphere in ``front'' in Figure~\ref{F:genus4chain}, which is a $(g+n+1)$--holed sphere containing the $n$ punctures of $S$, while $\Theta_o^-$ is the $(g+1)$--holed sphere in the ``back'' in Figure~\ref{F:genus4chain}. For every $A = \{ \alpha_i,\alpha_{i+1},\ldots,\alpha_j \}$ where $0 < i < j \leq 2g-1$ and $i,j$ are both odd, the boundary of the subsurface filled by $A$ has exactly two components $\beta_A^\pm$ with $\beta_A^+ \subset \Theta_o^+$ and $\beta_A^- \subset \Theta_o^-$ (possibly peripheral in $\Theta_o^\pm$ depending on $A$).  Furthermore, for every such set $A$, there is a ``complementary'' set $A' \subset \XX$ so that $A \cup A'$ is almost filling, and so that  $\{\beta_A^\pm\}$ is the set  determined by $A \cup A'$.  See Figure~\ref{F:Aodd}.

\begin{figure}[htb]
\labellist
\small\hair 2pt
 \pinlabel {$\beta_A^+$} [ ] at 390 65
 \pinlabel {$\beta_A^-$}  [ ] at 400 93
 \pinlabel {$A$} [ ] at 120 73
 \pinlabel {$A'$} [ ] at 132 27
\endlabellist
\begin{center}
\includegraphics[width=5in]{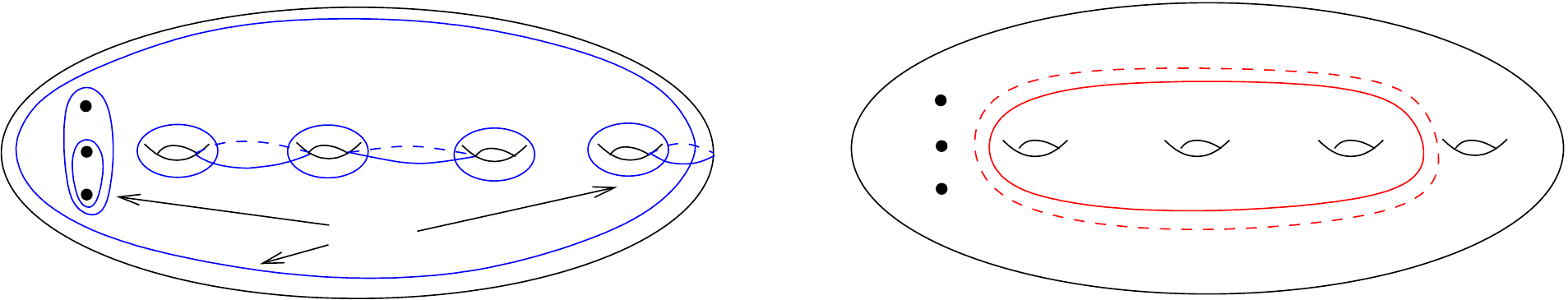} \caption{The sets $A = \{ \alpha_1,\ldots,\alpha_5 \} \subset \CC$ and $A' \subset \XX$ (left) and the curves $\beta_A^{\pm} \in \XX'$ determined by $A \cup A'$ (right).} \label{F:Aodd}
\end{center}
\end{figure}

\begin{lemma}
For all $g \geq 2$ and $n \geq 1$, the set $\XX_1$ is rigid and has trivial pointwise stabilizer in $\Mod^{\pm}(S_{g,n})$.
\label{L:X1}
\end{lemma}
\begin{proof} First, $\XX_1$ has trivial pointwise stabilizer since $\XX$ does. 
Given any locally injective simplicial map $\phi \colon \XX_1 \to \C(S)$, there exists a unique $h \in \Mod^{\pm}(S)$ so that $\phi = h|_{\XX'}$, by Theorem \ref{t:rigidhigher} and Proposition~\ref{p:prime}.  Composing with the inverse of $h$ if necessary, we can assume $\phi$ is the identity on $\XX'$. So we need only show that $\phi(\gamma) = \gamma$ for all $\gamma \in \XX_1 - \XX'$.   With respect to the notation above, any such curve is $\beta_A^{\pm}$ for  $A = \{ \alpha_i,\alpha_{i+1},\ldots,\alpha_j \}$, where $0 < i \leq j \leq 2g-1$ and $i,j$ are both odd. Since $A\cup A'$ is almost filling,  $\phi(\{\beta_A^{\pm}\}) = \{\beta_A^{\pm}\}$. Now, for $A=\{\alpha_1,\alpha_2, \alpha_3\}$, we have $i(\beta_A^+, \alpha_0^1) \ne 0$ and $i(\beta_A^-, \alpha_0^1) = 0$; here, $\alpha_0^1$ is the curve depicted in Figure \ref{F:alpha10}. Therefore $\phi(\beta_A^+) =\beta_A^+$, as $\phi$ is locally injective and simplicial. Finally, an easy connectivity argument involving the set of curves $\{\beta_A^{\pm}\}_A$ yields the desired result. 
\end{proof}

\subsection{Half the proof and the case of one or fewer punctures.}

We now enlarge the set $\XX_1 \subset \C(S)$ from Lemma \ref{L:X1} to $\XX_1^2=(\XX_1')' \subset \C(S)$.  According to Proposition~\ref{p:prime}, $\XX_1^2$ is rigid, and since the pointwise stabilizer of $\XX$ is trivial, so is the pointwise stabilizer of $\XX_1^2$. We will need the following lemma, see Figure~\ref{F:genus4chain} for the labeling of the curves:

\begin{lemma}
For any $g \geq 2$ and $n \geq 0$ (with $n \geq 1$ if $g = 2$), we have $T_{\alpha_{2g}}(\alpha_{2g-1}) \in \XX_1^2$.
\label{L:determine_chain}
\end{lemma}

\begin{proof}
This requires a series of pictures, slightly different for the case $g \geq 3$ and for $g = 2$.  

\smallskip

\noindent{\bf Case 1: $g \geq 3$.} We refer the reader to Figure~\ref{F:chain_g>2}: although we have only drawn the figures for $g=3$ and $n=2$, it is straightforward to extend them to all $g \geq 3$ and $n \geq 0$.
The upper left figure shows an almost filling set of curves contained in $\XX_1$, determining uniquely the curve on the upper right figure, which is thus in $\XX_1'$. This curve is then used to produce an almost filling set, depicted on the lower left hand figure, that uniquely determines $T_{\alpha_{2g}}(\alpha_{2g-1})$, shown on the right. Thus we see that $T_{\alpha_{2g}}(\alpha_{2g-1})\in \XX^2_1$, as claimed. 

\smallskip

\noindent{\bf Case 2: $g=2$ and $n\ge 1$.} In this case a different set of pictures is required, see Figure \ref{F:2n_chain_chain}. The upper left hand figure shows an almost filling set of curves that is contained in $\XX_1$ and uniquely determines the curve shown on the upper right. This curve is then used to produce an almost filling set, depicted in the middle left picture, which is contained in $\XX'_1$ and uniquely determines the curve in the middle right figure. We now make use of this new curve to produce an almost filling set (lower left)  that is contained in $\XX_1^2$ and uniquely determines $T_{\alpha_{2g}}(\alpha_{2g-1})$ (lower right). 
\end{proof}

\begin{figure}[htb]
\begin{center}
\includegraphics[width=5in]{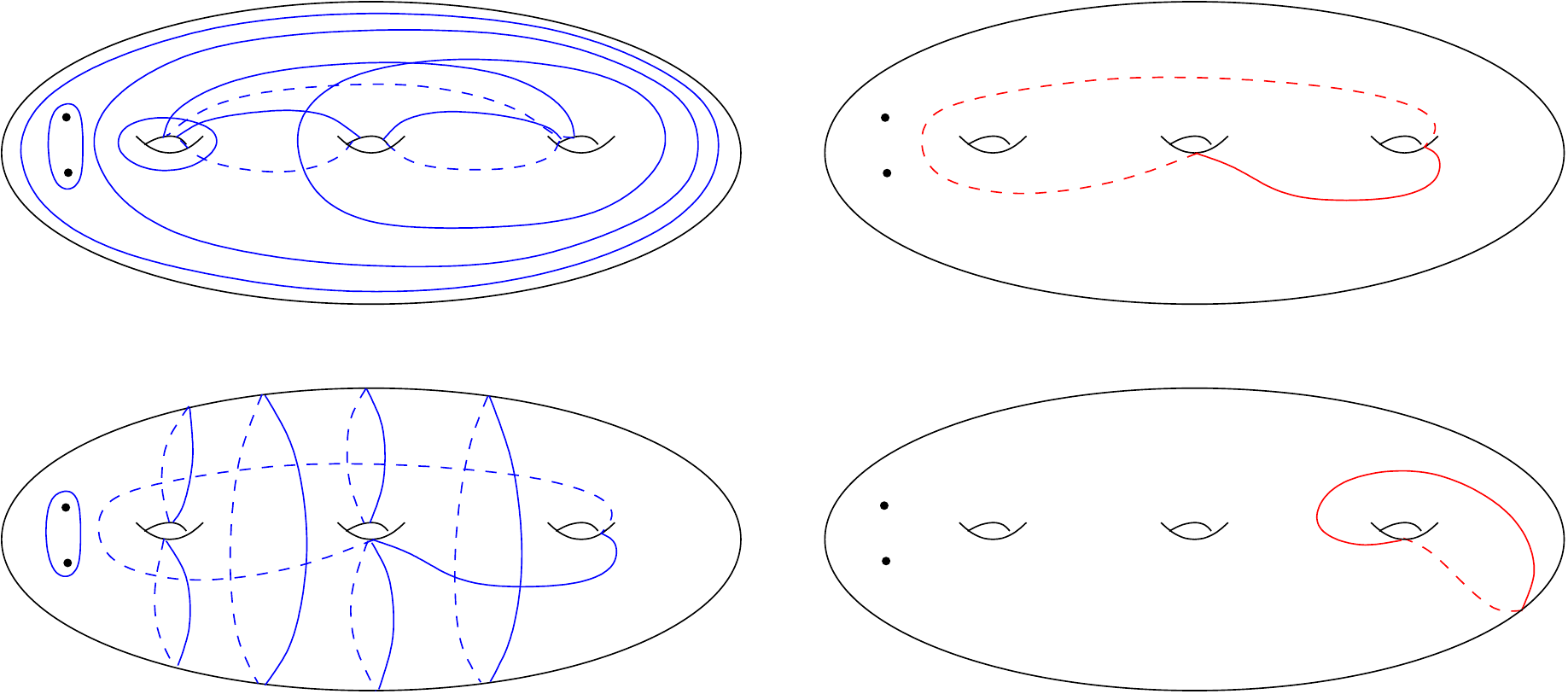} \caption{Illustrating $T_{\alpha_{2g}}(\alpha_{2g-1})$ in $\XX_1^2$, when $g=3$. The almost filling set on the left (blue) uniquely determines the curve in the right figure (red).}\label{F:chain_g>2}
\end{center}
\end{figure}

\begin{figure}[htb]
\begin{center}
\includegraphics[width=5in]{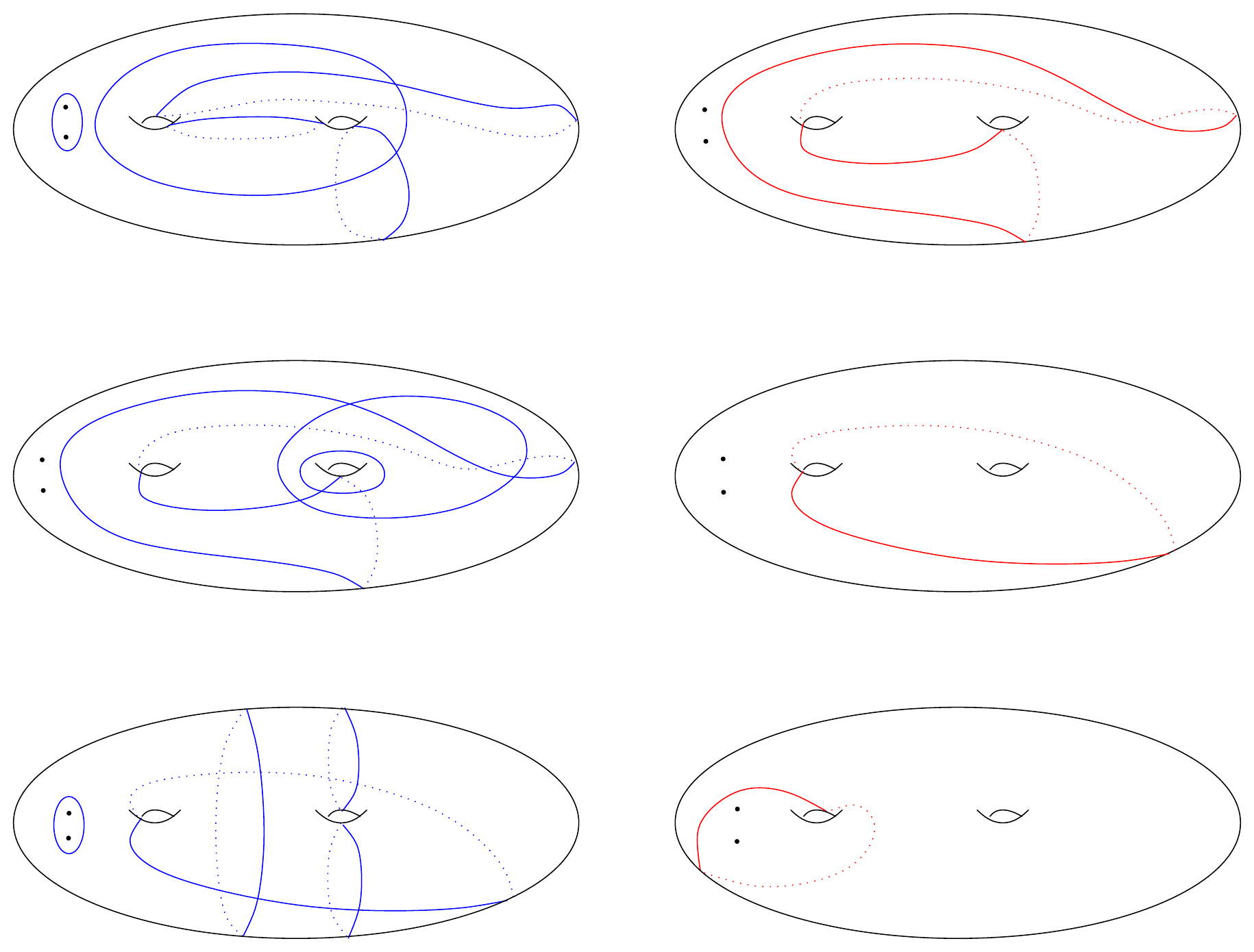} \caption{Illustrating $T_{\alpha_{2g}}(\alpha_{2g-1})$ in $\XX_1^2$, when $g=2$ and $n\ge 1$. The almost filling set on the left (blue) uniquely determines the curve in the right figure (red).}\label{F:2n_chain_chain}
\end{center}
\end{figure}

We claim that the set $\XX_1^2 \cup T_\CC(\CC)$ is rigid. More concretely: 

\begin{lemma}
Let $\phi: \XX_1^2 \cup T_\CC(\CC)\to \C(S)$ be a locally injective simplicial map. Then there exists a unique $h\in \Mod^{\pm}(S)$ such that $h|_{\XX_1^2 \cup T_\CC(\CC)} =\phi$.
\label{L:enlarge1}
\end{lemma}

\begin{proof}
Let $\phi: \XX_1^2 \cup T_\CC(\CC)$ be a locally injective simplicial map.  Since $\XX_1^2$ is rigid and its pointwise stabilizer in $\Mod^{\pm}(S)$ is trivial, there exists a unique $h\in \Mod^{\pm}(S)$ such that $h|_{\XX_1^2} =\phi|_{\XX_1^2} $. Precomposing $\phi$ with $h^{-1}$, we may assume that in fact $\phi|_{\XX_1^2} $ is the identity map.

For $i=0, \ldots, n$, $\CC_i=\{\alpha_0^i, \alpha_1, \ldots, \alpha_{2g+1}\}$ is a closed string of twistable Farey neighbors in $\XX_1^2$ (the fact that the nonzero intersection numbers between these curves is $\XX$--detectable, hence $\XX_1^2$--detectable, is shown in the proofs of Theorem 5.1 and 6.1 in \cite{AL}). Consider the set $\XX_1^2 \cup T_{\CC_i}(\CC_i)$, and observe that,  in the terminology of Definition \ref{D:strings}, it equals $Y_A$ for $Y= \XX_1^2$ and $A = \CC_i$. By Proposition \ref{P:closed_string}, $$\phi(\XX_1^2 \cup T_{\CC_i}(\CC_i)) = \XX_1^2 \cup T_{\CC_i}(\CC_i);$$ moreover, the automorphism group of $\XX_1^2 \cup T_{\CC_i}(\CC_i)$ fixing $\XX_1^2$ pointwise has order at most two. But, by Lemma \ref{L:determine_chain}, $T_{\alpha_{2g}}(\alpha_{2g-1}) \in \XX_1^2$ and thus such group is trivial. In other words, we have shown that the set $\XX_1^2 \cup T_{\CC_i}(\CC_i)$ is rigid. 

Now, $\XX_1^2 \cup T_{\CC_0}(\CC_0) \cup T_{\CC_1}(\CC_1)$ is also rigid by Lemma \ref{L:glue}, since $(\XX_1^2 \cup T_{\CC_0}(\CC_0) ) \cap (\XX_1^2 \cup T_{\CC_1}(\CC_1))$
is weakly rigid as it contains $\XX_1^2$. 
Since $ T_\CC(\CC) = \bigcup_{i=0}^n T_{\CC_i}(\CC_i)$, we may repeat essentially this same argument $n-1$ more times to conclude $\XX_1^2 \cup T_{\CC}(\CC)$ is rigid, as required.
\end{proof}

Next, we provide a further enlargement of our rigid set. Let $\beta$ be the curve depicted in Figure \ref{F:examples},  which is one of the boundary components of the surface $N(\alpha_{2g-2} \cup \alpha_{2g-1} \cup \alpha_{2g})$. We claim:

\begin{lemma}
The set $\XX_1^2 \cup T_\CC(\CC) \cup T_\beta(\CC)$ is rigid. 
\label{L:enlarge2}
\end{lemma}

\begin{proof}
Let $\phi: \XX_1^2 \cup T_\CC(\CC) \cup T_\beta(\CC) \to \C(S)$ be a locally injective simplicial map. By Lemma \ref{L:enlarge1}, $\XX_1^2 \cup T_\CC(\CC)$ is rigid and thus, up to precomposing $\phi$ with an element of $\Mod^{\pm}(S)$, we may assume that $\phi|_{\XX_1^2 \cup T_\CC(\CC)}$ is the identity. The set $$A= \{\alpha_{2g}, \alpha_{2g-1}, \alpha_{2g-2}, \beta, \alpha_{2g+1}\}\subset \XX_1^2$$ is a closed string of twistable Farey neighbors in $\XX_1^2$ (again, detectability of the nonzero intersection numbers is shown in \cite{AL}).  Therefore, we may apply Proposition \ref{P:closed_string} to $\XX=  \XX_1^2 \cup T_\CC(\CC)$ and $A$ to deduce that $\phi(\XX_A) = \XX_A$; observe that $\XX_A = \XX_1^2 \cup T_\CC(\CC) \cup T_\beta(\CC)$. Moreover, the automorphism group of $\XX_A$ fixing $\XX$ pointwise is trivial, by Lemma \ref{L:enlarge1}, and thus the result follows. 
\end{proof}

\begin{proof}[Proof of Theorem~\ref{main} for $g \geq 2$ and $n \leq 1$]  Let $Y = \XX_1^2 \cup T_{\CC}(\CC) \cup T_{\beta}(\CC)$.
When $S$ is closed or has one puncture, the Dehn twists about chain curves and the Dehn twist about the curve $\beta$  generate $\Mod(S)$; see, for example, Corollary 4.15 of \cite{Farb-Margalit}.  For $\gamma \in \CC \cup \{ \beta \}$, the set
\[ T_\gamma(Y) \cap (Y) \]
contains $\CC$, together with $T_\alpha(\alpha')$ for any $\alpha,\alpha' \in \CC$ which are disjoint from $\gamma$.  In particular, this set is weakly rigid.  By inspection, the $\Mod(S)$--orbit of $Y$ is all of $\C(S)$, and so by Proposition~\ref{P:finish}, this set suffices to prove the theorem.
\end{proof}

\subsection{Multiple punctures.}
When $S_{g,n}$ has $n \geq 2$ (and $g \geq 2$), the twists in the curves $\CC$ and $\{\beta\}$ do not generate the entire mapping class group.  In this case, one needs to add the set of half-twists about the outer curves $\epsilon^{i \, i+1}$ bounding twice-punctured disks; see again Corollary 4.15 of \cite{Farb-Margalit}. Because of this, and in the light of Proposition~\ref{P:finish}, when $n \geq 2$ we would like to enlarge our rigid set from the previous subsection by adding half-twists of chain curves about outer curves $\epsilon^{i \, i+1}$. In fact, denoting this set of outer curves by $\mathfrak O_P = \{\epsilon^{i \, i+1}\}_{i=1}^{n-1}$ we shall show that these curves are already in $\XX_1^2$.  Specifically, we prove

\begin{lemma}
We have $H_{\mathfrak O_P}(\CC) \subset \XX_1^2$. 
\label{L:enlarge3}
\end{lemma}

\begin{proof}  If $\alpha \in \CC$ and $\epsilon^{j \, j+1} \in \mathfrak O_P$, then we must show that $H_{\epsilon^{j \, j+1}}(\alpha) \in \XX_1^2$ for each $j = 1,\ldots, n-1$.  This is clear if $i(\alpha,\epsilon^{j \, j+1}) = 0$, since then $H_{\epsilon^{j \, j+1}}(\alpha) = \alpha$.  The intersection number is nonzero only when $\alpha = \alpha_0^j$, so it suffices to consider only this case.

To prove $H_{\epsilon^{j \, j+1}}(\alpha_0^j) \in \XX_1^2$, we need only exhibit the almost filling sets from $\XX_1'$ uniquely determining this curve.  This in turn requires an almost filling set from $\XX_1$.  As before, we provide the necessary curves in a sequence of two figures.
First, the almost filling set on the left of Figure \ref{F:half1} is contained in $\XX'$, and hence $\XX_1$ (compare with Figure \ref{F:alpha1j}), and uniquely determines the curve $\gamma_1$ depicted on the right of the same figure.  Therefore, $\gamma_1 \in \XX_1'$.  Figure \ref{F:half2} is then an almost filling set in $\XX_1'$, and uniquely determines the curve on the right of the same figure.  This curve is $H_{\epsilon^{j \, j+1}}(\alpha_0^j)$, and so completes the proof.
\end{proof}

\begin{figure}[htb]
\begin{center}
\includegraphics[width=5in]{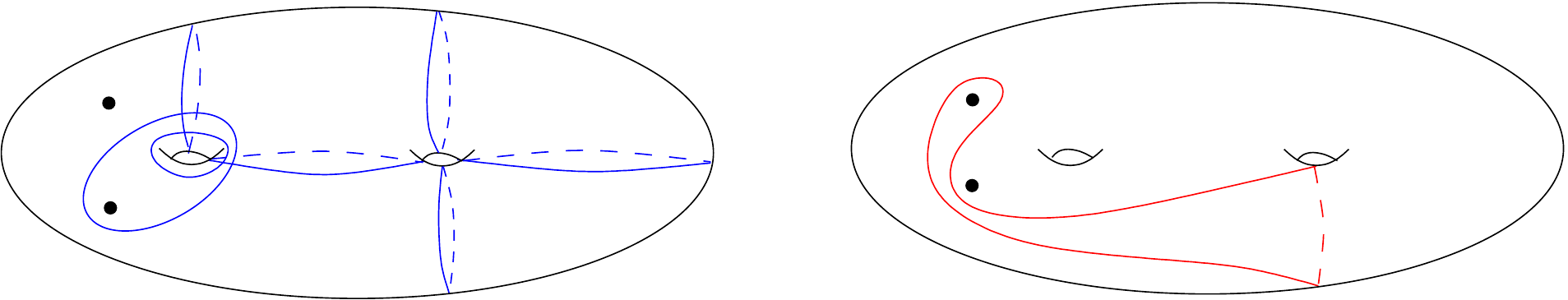} \caption{Determining the curve $\gamma_1$.}\label{F:half1}
\end{center}
\end{figure}

\begin{figure}[htb]
\begin{center}
\includegraphics[width=5in]{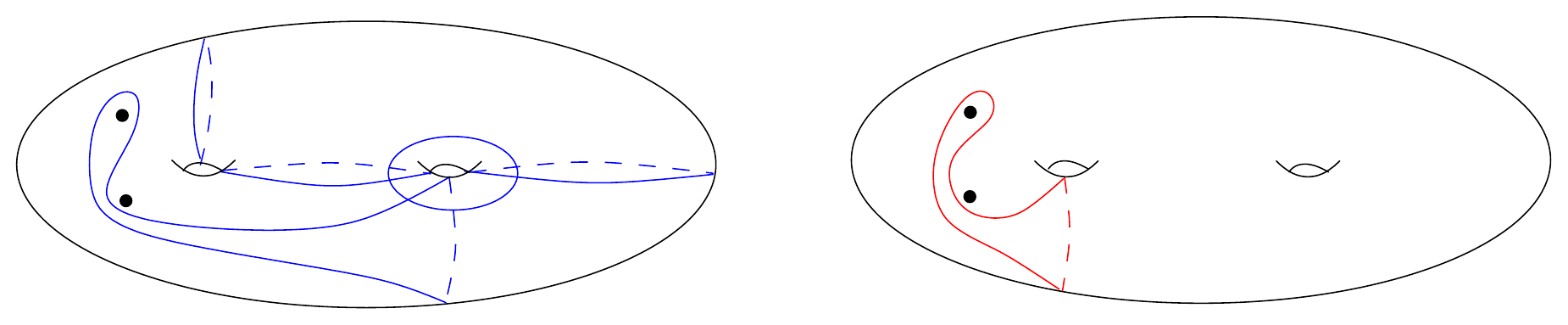} \caption{Determining the curve $H_{\epsilon^{i,i+1}}(\alpha_0^i)$.}\label{F:half2}
\end{center}
\end{figure}

We are finally in a position to prove Theorem \ref{main} for surfaces of genus $g\ge 2$ and $n \geq 2$:

\begin{proof}[Proof of Theorem \ref{main} for $S=S_{g,n}$, $g\ge 2$, $n \geq 2$.]
The set $Y= \XX_1^2 \cup T_\CC(\CC) \cup T_\beta(\CC)$ is rigid by Lemma \ref{L:enlarge3}, and has trivial pointwise stabilizer in $\Mod^{\pm}(S)$ since $\XX$ does.  Moreover $\Mod(S) \cdot Y = \C(S)$, by inspection.  Consider the subset $ G = \CC \cup \{\beta\} \cup \mathfrak O_P$; as mentioned before, the (half) twists about elements of  $G$ generate $\Mod(S)$.  In addition, for every $\alpha \in  G$, $Y \cap f_\alpha(Y)$ is weakly rigid.  Thus we can apply Proposition \ref{P:finish} to $Y$ and $ G$, hence obtaining the desired exhaustion of $\C(S)$.  
\end{proof}

\section{Tori}

In this section we will prove Therorem \ref{main} for $S=S_{1,n}$, for $n\ge 0$. First, if $n \le 1$ then $\C(S)$ is isomorphic to the Farey complex, and thus the result follows as in the case of $S_{0,4}$; see Section \ref{S:punctured spheres}. For $n=2$, Theorem \ref{main} is not true as stated due to the existence of {\em non-geometric} automorphisms of $\C(S)$, as mentioned in the introduction. However, in the light of the isomorphism $\C(S_{0,5}) \cong \C(S_{1,2})$ \cite{Luo}, the same statement holds after replacing the group $\Mod^{\pm}(S)$ by $\Aut(\C(S))$ in the definition of rigid set, by the results of Section \ref{S:punctured spheres}.

Therefore, from now on we assume $n\ge 3$.  In \cite{AL}, we constructed a finite rigid set $\XX$ described as follows.  View $S_{1,n}$ as a unit square with $n$ punctures along the horizonal midline and the sides identified.  The set $\XX$ contains a subset $\CC \subset \XX$ of $n+1$ {\em chain curves}
\[ \CC = \{ \alpha_1,\ldots,\alpha_n\} \cup \{\beta\} \]
where $\alpha_1,\ldots,\alpha_n$ are distinct curves which appear as vertical lines in the square and $\beta$ is the curve which appears as a horizontal line; see Figure \ref{F:original torus curves}.  We assume that the indices on the $\alpha_i$ are ordered cyclically around the torus, and that the punctures are labelled so that the $i^{th}$ puncture lies between $\alpha_i$ and $\alpha_{i+1}$.  The boundaries of the subsurfaces filled by connected unions of these chain curves form a collection of curves, denoted $\OO$ which we refer to as {\em outer curves}.   Then $$\XX = \CC \cup \OO.$$  This set has a nontrivial pointwise stabilizer in $ \Mod^\pm(S_{1,n})$, which can be realized as the (descent to $S_{1,n}$ of the) horizontal reflection of the square through the midline containing the punctures.  Denoting this involution $r \colon S_{1,n} \to S_{1,n}$, we can summarize the result of \cite{AL} in the following:

\begin{figure}[htb]
\begin{center}
\includegraphics[width=4in]{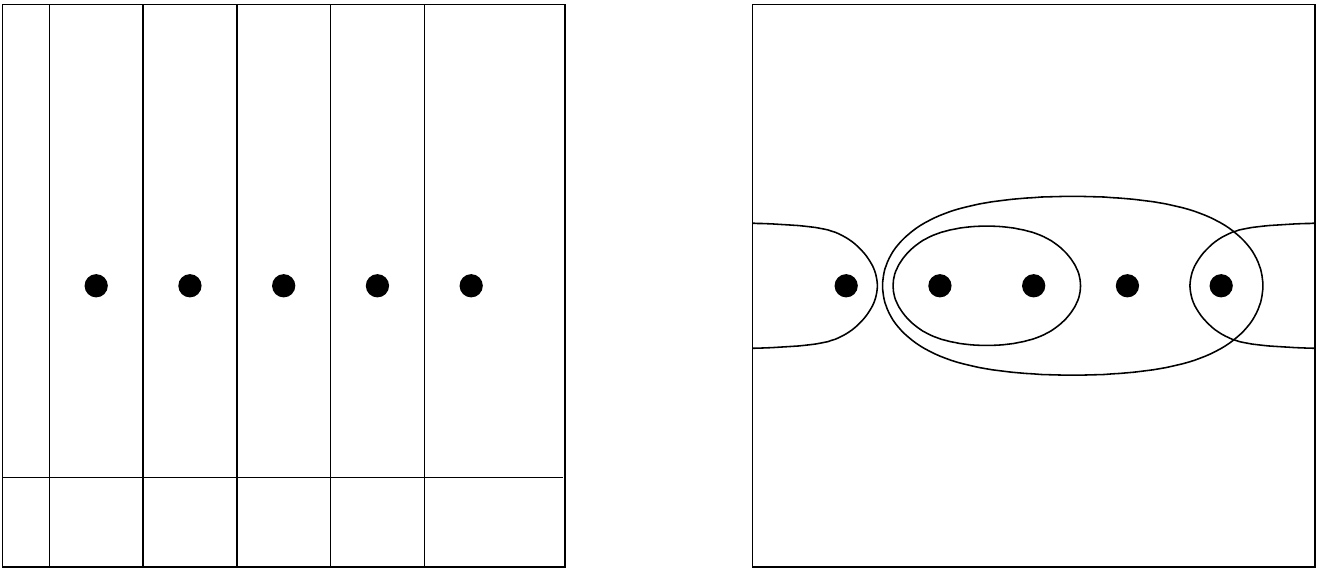} \caption{Chain curves on the left, and some examples of outer curves on the right, in $S_{1,5}$.}\label{F:original torus curves}
\end{center}
\end{figure}

\begin{theorem} \cite{AL}
For any locally injective simplicial map $\phi \colon \XX \to \C(S_{1,n})$ there exists $h \in \Mod^\pm(S_{1,n})$ such that $h|_{\XX} = \phi$.  Moreover, $h$ is unique up to precomposing with $r$.
\label{T:AL_torus}
\end{theorem}

The strategy of proof is again similar to that of previous sections, although the technicalities are different, and boils down to producing an enlargement of the set $\XX$ so that Proposition \ref{P:finish} can be applied. 

We begin by enlarging the set $\XX$ as follows.  We let $\delta_i$ be the curve coming from the vertical line through the $i^{th}$ puncture in the square.  For every $1 \leq i \leq n$, let $\beta_i^+$ be the curve obtained from $\beta$ by pushing it up over the $i^{th}$ puncture.  More precisely, we consider the point-pushing homeomorphism $f_i \colon S_{1,n} \to S_{1,n}$ that pushes the $i^{th}$ puncture up and around $\delta_i$, and then let $\beta_i^+ = f_i(\beta)$.   We similarly define $\beta_i^- = f_i^{-1}(\beta)$, and set $\beta_{i,i+1}^\pm = f_{i+1}^{\pm 1}f_i^{\pm 1}(\beta)$, where the subscripts are taken modulo $n$.  See Figure~\ref{F:torus beta curves}.

\begin{figure}[htb]
\begin{center}
\includegraphics[width=1.8in]{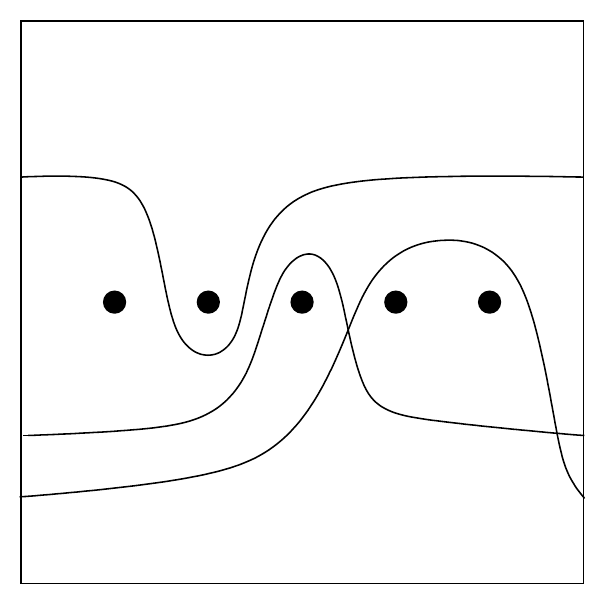} \caption{Curves $\beta_2^-$, $\beta_3^+$, and $\beta_{4,5}^+$ on $S_{1,5}$.}\label{F:torus beta curves}
\end{center}
\end{figure}

Let
\[ \XX_1 = \XX \cup \{\beta_i^\pm \mid 1 \leq i \leq n \} \cup \{ \beta_{i,i+1}^\pm \mid 1 \leq i \leq n\} \]
with indices in the last set taken modulo $n$.  We first prove that this set is rigid; since the poinwise stabilizer of $\XX_1$ in $\Mod^{\pm}(S_{1,n})$ is trivial, this amounts to the following: 

\begin{proposition}  For any locally injective simplicial map $\phi \colon \XX_1 \to \C(S_{1,n})$, there exists a unique $h \in \Mod^\pm(S_{1,n})$ so that $h|_{\XX_1} = \phi$.
\label{P:torus_large}
\end{proposition}

The proof of this proposition will require a repeated application of Lemma \ref{l:fareydetect}, and as such, we must verify that certain quadruples of curves satisfy the hypotheses of that lemma.  We will need to refer to the outer curves by name.  To this end, note that since any outer curve surrounds a set of (cyclically) consecutive punctures, we can determine an outer curve by specifying the first and last puncture surrounded.  Consequently, we let $\epsilon^{i \, j}$ denote the outer curve surrounding all punctures from the $i^{th}$ to the $j^{th}$, with all indices taken modulo $n$.   
Observe that since the set of punctures is cyclically ordered, we do not need to assume that $i<j$ in the definition of $\epsilon^{i \, j}$. 
We will need the following lemma:

\begin{figure}[htb]
\begin{center}
\includegraphics[width=4in]{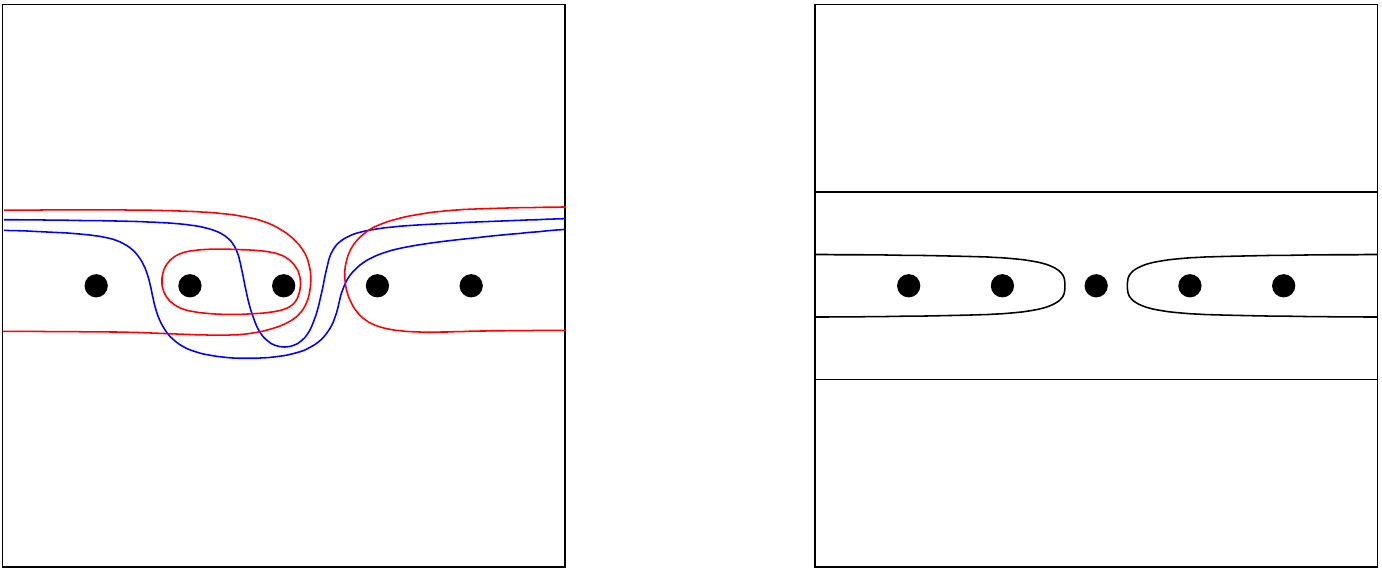} \caption{The curves $\beta_{2 \, 3}^-,\epsilon^{4 \, 3},\beta_3^-,\epsilon^{2 \, 3}$ on the left.  The four-holed sphere filled by $\epsilon^{4 \, 3}$ and $\beta_3^\pm$ on the right.}\label{F:torusfareynb2}
\end{center}
\end{figure}

\begin{lemma} \label{L:quads_satisfying_Lfareydetect}
For each $1 \leq i \leq n$, consider the following four quadruples of curves in $\XX_1$, with indices taken modulo $n$:
\begin{itemize}
\item $\beta_{i-1 \, i}^\pm,\epsilon^{i+1 \, i},\beta_i^\pm,\epsilon^{i-1 \, i}$,
\item $\beta_{i \, i+1}^\pm,\epsilon^{i \, i-1}, \beta_i^\pm,\epsilon^{i \, i+1}$,
\end{itemize}
Each of these satisfies the hypothesis of Lemma~\ref{l:fareydetect}.  Furthermore, the nonzero intersections are all $\XX_1$--detectable.  Consequently, $\epsilon^{i+1 \, i}$ and $\epsilon^{i \, i-1}$ are the unique Farey neighbors of $\beta_i^-$ and $\beta_i^+$.
\end{lemma}
\begin{proof}
The fact that the four quadruples of curves each satisfy the hypothesis of Lemma~\ref{l:fareydetect} is clear by inspection.  See the left side of Figure~\ref{F:torusfareynb2} for the case $\beta_{i-1 \, i}^-,\epsilon^{i+1 \, i},\beta_i^-,\epsilon^{i-1 \, i}$.  The four-holed sphere $N$ filled by the Farey neighbors $\epsilon^{i+1 \, i},\beta_i^\pm$ and $\epsilon^{i \, i-1}, \beta_i^\pm$ has holes corresponding to the $i^{th}$ puncture, and the curves $\beta$ and $\epsilon^{i+1 \, i-1}$; see the right side of Figure~\ref{F:torusfareynb2}.  Only $\epsilon^{i+1 \, i-1}$ intersects $\beta_{i-1 \, i}^\pm, \beta_{i \, i+1}^\pm, \epsilon^{i-1 \, i},\epsilon^{i \, i+1}$ nontrivially, as required for Lemma~\ref{l:fareydetect}.

To see that all the intersections are $\XX_1$--detectable, we need only exhibit the necessary curves in $\XX_1$ determining a pants decomposition of $S - N$.  See Figure~\ref{F:torusdetectfareynb1} for the curves necessary to detect $i(\beta_i^-,\epsilon^{i-1 \, i}) \neq 0$.  We leave the other cases to the reader.
\begin{figure}[htb]
\begin{center}
\includegraphics[width=1.8in]{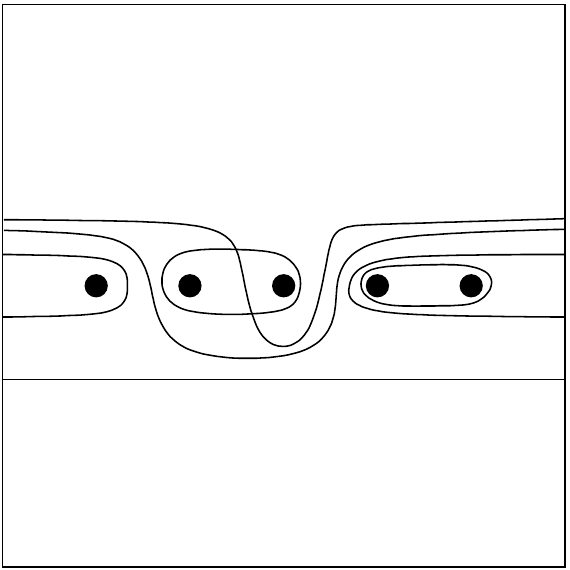} \caption{We use $\{\beta, \beta_{2 \, 3}^-,\epsilon^{4 \, 5},\epsilon^{4 \, 1}\}$ to detect $i(\beta_3^-,\epsilon^{2 \, 3}) \neq 0$.}\label{F:torusdetectfareynb1}
\end{center}
\end{figure}
\end{proof}

We are now in a position to prove Proposition \ref{P:torus_large}:

\begin{proof}[Proof of Proposition \ref{P:torus_large}]
Let $\phi \colon \XX_1 \to \C(S_{1,n})$ be a locally injective simplicial map. By Theorem \ref{T:AL_torus},  there exists $f \in \Mod^\pm(S_{1,n})$ such that $f|_{\XX} = \phi|_{\XX}$, unique up to precomposing with $r$. In fact, after precomposing $\phi$ with $f^{-1}$ we may as well assume that $\phi|_{\XX}$ is the identity.  

According to Lemma~\ref{L:quads_satisfying_Lfareydetect}, for all $i$, $\phi(\epsilon^{i+1 \, i}) = \epsilon^{i+1 \, i}$ and $\phi(\epsilon^{i \, i-1}) = \epsilon^{i \, i-1}$ are the unique Farey neighbors of $\phi(\beta_i^-)$ and $\phi(\beta_i^+)$ (with indices taken modulo $n$).  Consequently, $\phi(\{\beta_i^{\pm}\}) = \{ \beta_i^\pm\}$ for all $i$.  Notice that $i(\beta_i^+,\beta_j^-) = 0$ for all $i,j$, while $i(\beta_i^+,\beta_j^+) = i(\beta_i^-,\beta_j^-) = 2$ for all $i,j$.  It follows that if $\phi(\beta_i^-) = \beta_i^+$ for some $i$, then this is true for all $i$.  Composing with $r$ if necessary, we deduce that $\phi(\beta_i^\pm) = \beta_i^\pm$ for all $i$.  All that remains is to see that $\phi(\beta_{i \, i+1}^\pm) = \beta_{i \, i+1}^\pm$ for all $i$.

To prove this we need only show that $\beta_{i \, i+1}^\pm \in (\XX \cup \{\beta_j^\pm \mid 1 \leq j \leq n \})'$, and then we can apply Proposition~\ref{p:prime}.  First note that when $n = 3$, then $\beta_{i \, i+1}^\pm = \beta_{i+2}^\mp$, so there is nothing to prove in this case.  In general, one readily checks that $\beta_{i \, i+1}^+$ is uniquely determined by the almost filling set
\[ \{ \beta,\beta_1^-,\beta_2^-,\ldots,\beta_n^- \} \setminus \{ \beta_i^-,\beta_{i+1}^- \}.\]
This completes the proof.
\end{proof}

Let $\OO_P = \{\epsilon^{i \, i+1}\}_{i=1}^n$, counting indices modulo $n$.  For $n \geq 5$, this is a closed string of twistable Farey neighbors in $\XX_1$, and we could appeal to Proposition~\ref{P:closed_string} to add the half-twists about curves of $\OO_P$ about curves in $\OO_P$ in this case.  However, we can provide a single argument for all $n \geq 3$.

\begin{lemma} \label{L:half_twist_tori}
For all $\epsilon,\epsilon' \in \OO_P$, $H_{\epsilon}^{\pm 1}(\epsilon') \in \XX'$.  Consequently, $H_{\epsilon}^{\pm 1}(\XX_1') \cup \XX_1'$ is rigid.
\end{lemma}
\begin{proof}
We start with the proof of the first statement.  If $i(\epsilon,\epsilon') = 0$, then there is nothing to prove.  Otherwise, up to a homeomorphism we may assume that $\epsilon = \epsilon^{i \, i+1}$ and $\epsilon' = \epsilon^{i+1 \, i+2}$.  Then we note that $H_{\epsilon^{i \, i+1}}(\epsilon^{i+1 \, i+2})$ is the curve uniquely determined by the almost filling set of curves
\[ \{ \beta_{i+1}^+ \} \cup \{\alpha_1,\ldots,\alpha_n \} \setminus \{ \alpha_{i+1},\alpha_{i+2} \},\]
completing the proof of the first statement.

For the second statement, we note that $H_{\epsilon^{i \, i+1}}(\XX_1') \cap \XX_1'$ contains the weakly rigid set $\OO_P \cup \{\beta_{i+2}^+ \}$, for example.  Therefore, since $\XX_1$ is rigid by Proposition~\ref{P:torus_large}, so is $\XX_1'$ by Proposition~\ref{p:prime}, and hence by Lemma~\ref{L:glue} it follows that $H_{\epsilon^{i \, i+1}}(\XX_1') \cup \XX_1'$ is rigid, as required.  A similar argument proves the statement for $H_{\epsilon^{i \, i+1}}^{-1}$.
\end{proof}

We also need to consider Dehn twists in $\alpha_i$ and $\beta$.  To deal with these, we first define $\XX_2 = \XX_1' \cup H_{\OO_P}(\XX_1')$, where $H_{\OO_P}(\XX_1')$ is the union of $H_{\epsilon}^{\pm 1}(\XX_1')$ over all $\epsilon \in \OO_P$.  By Lemma~\ref{L:half_twist_tori}, $\XX_2$ is rigid.

\begin{lemma} 
For all $i = 1,\ldots,n$, we have $T_{\alpha_i}^{\pm 1}(\beta) = T_{\beta}^{\mp 1}(\alpha_i) \in \XX_1^2 \subset \XX_2^2$.  Consequently, $T_{\alpha_i}^{\pm 1}(\XX_2^2) \cup \XX_2^2$ and $T_\beta^{\pm 1}(\XX_2^2) \cup \XX_2^2$ is rigid.
\end{lemma}
\begin{proof}
As in previous arguments,  we exhibit a series of pictures that will yield the desired result; see Figure \ref{F:torustwist}.  It is straightforward to modify such pictures to treat the case of an arbitrary $n\ge 3$.  The top left picture shows an almost filling set in $\XX_1$ that uniquely determines a curve in $\XX_1'$ on the top right.  Then the lower left is an almost filling set in $\XX_1'$ that uniquely determines the curve in $(\XX_1')' = \XX_1^2$.  This curve is precisely $T_\beta(\alpha_i) = T_{\alpha_i}^{-1}(\beta)$.  Similarly, $T_{\alpha_i}(\beta) = T_\beta^{-1}(\alpha_i) \in \XX_1^2$.

Finally, we easily observe that $\XX_2^2 \cap T_{\alpha_i}(\XX_2^2)$ is weakly rigid, as it contains $\CC \cup H_{\epsilon^{i-1 \, i}}(\alpha_{i-1})$, which is weakly rigid.  Appealing to Lemma~\ref{L:glue}, it follows that $\XX_2^2 \cup T_{\alpha_i}(\XX_2^2)$ is rigid.  The other cases follow similarly.
\end{proof}

\begin{figure}[htb]
\begin{center}
\includegraphics[width=4in]{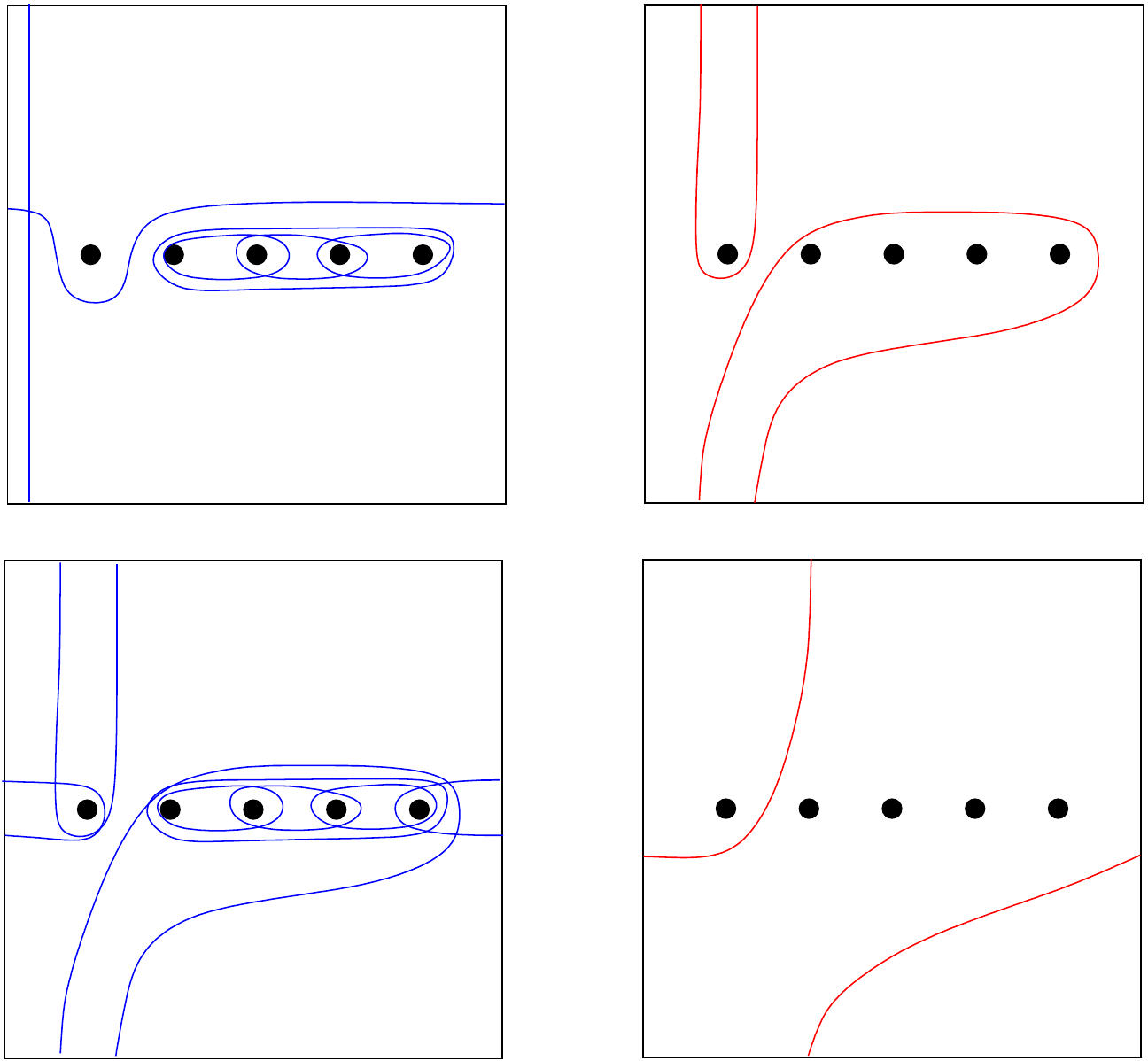} \caption{Illustrating $T_\beta(\alpha_2) \in \XX_2^2$ on $S_{1,5}$. }\label{F:torustwist}
\end{center}
\end{figure}

Finally, we prove our main result for surfaces of genus 1.

\begin{proof}[Proof of Theorem \ref{main} for $S=S_{1,n}$]
Since $\XX_2$ is rigid, by Propositions \ref{p:prime}, the set $Y=\XX_2^2$ is rigid.  
Moreover, $\Mod(S_{1,n}) \cdot Y = \C(S_{1,n})$, by inspection.  A generating set for $\Mod(S_{1,n})$ is given by the Dehn twists $f_\alpha$ about the elements $\alpha \in \CC$ and the half-twists $f_\epsilon$ about the elements $\epsilon \in A= \{\epsilon^{i \, i+1}\}$ (see Section 4.4 of \cite{Farb-Margalit}, for instance). Let $G= \CC \cup A$ and note that, for each $\alpha \in G$, the set $Y \cup f_\alpha Y$ is rigid.  Therefore, we may apply Proposition \ref{P:finish} to obtain the desired exhaustion of $\C(S_{1,n})$ by finite rigid sets. 
\end{proof}

\end{document}